\documentclass{amsart}
\usepackage{amsmath}
\usepackage{amssymb}
\usepackage{amsthm}
\usepackage{amsfonts}
\usepackage{url}
\usepackage{hyperref}
\usepackage{appendix}
\usepackage{verbatim}
\usepackage{graphicx,color}
\usepackage[pagewise]{lineno}
\usepackage{tikz-cd}
\usepackage[curve,tips]{xypic}
\SelectTips{eu}{11}
\UseTips
\usepackage{tikz}
\usepackage{pst-knot}
\usepackage{bold-extra}

\newtheorem{theorem}{Theorem}[section]
\newtheorem{lemma}[theorem]{Lemma}
\newtheorem{corollary}[theorem]{Corollary}
\newtheorem{proposition}[theorem]{Proposition}

\theoremstyle{definition} 
\newtheorem{definition}[theorem]{Definition}
\newtheorem{remark}[theorem]{Remark}
\newtheorem{example}[theorem]{Example}

\title{Absolute profinite rigidity and hyperbolic geometry}

\author{M. R. Bridson}
\author{D. B. McReynolds}
\author{A. W. Reid}
\author{R. Spitler}
\address{\newline Mathematical Institute, 
\newline Andrew Wiles Building,
\newline University of Oxford,
\newline Oxford OX2 6GG, UK}
\email{ bridson@maths.ox.ac.uk}
\address{\newline Department of Mathematics,
\newline Purdue University
\newline West Lafayette, IN 47907, USA}
\email{dmcreyno@purdue.edu}
\address{\newline Department of Mathematics,
\newline Rice University, 
\newline Houston, TX 77005, USA}
\email{ alan.reid@rice.edu}
\address{\newline Department of Mathematics,
\newline McMaster University
\newline Hamilton, Ontario, Canada}
\email{spitlerr@mcmaster.ca}

\def\-{\overline}

\def\wh{\widehat}
\def\g{\gamma}
\def\G{\Gamma}

\def\H{\mathbb{H}}
\def\Z{\mathbb{Z}}
\def\R{\mathbb{R}}
\def\Q{\mathbb{Q}} 
\def\F{\mathbb{F}} 
\def\C{\mathbb{C}} 
\def\Out{{\rm{Out}}}
\def\Mod{{\rm{Mod}}}

\def\ab{\rm{ab}}

\def\mod{\rm{mod}}

\DeclareMathOperator{\SL}{SL} 
\DeclareMathOperator{\PSL}{PSL}
 
\DeclareMathOperator{\PGL}{PGL}

\DeclareMathOperator{\PSU}{PSU}

\def\Epi{\rm{Epi}}

\def\Isom{\rm{Isom}}
\def\qed{ $\sqcup\!\!\!\!\sqcap$}

\def\tr{\mbox{\rm{tr}}\, }
\def\P{\mbox{\rm{P}}}

\def\G{\Gamma}
\def\Gsis{\Gamma_s}
\def\Lsis{L_s}
\def\Deltasis{\Delta_s}
 
\def\<{\langle}
\def\>{\rangle}
\def\ilim{\varprojlim}
\def\ns{\vartriangleleft}

\def\Epi{{\rm{Epi}}}
\def\onto{\twoheadrightarrow}
\newcommand{\innp}[1]{\left< #1 \right>}
\newcommand{\abs}[1]{\left\vert#1\right\vert}
\newcommand{\set}[1]{\left\{#1\right\}}

\DeclareMathOperator{\Aut}{Aut}

\begin{document}

\date{}

\begin{abstract} 
We construct arithmetic Kleinian groups that are profinitely rigid in the absolute sense: each is distinguished from all other finitely generated, residually finite groups by its set of finite quotients. The Bianchi group $\PSL(2,\Z[\omega])$ with $\omega^2+\omega+1=0$ is rigid in this sense. Other examples include the non-uniform lattice of minimal co-volume in ${\rm{PSL}}(2,\C)$ and the fundamental group of the Weeks manifold (the closed hyperbolic $3$--manifold of minimal volume).
\end{abstract}


\maketitle

\section{Introduction} \label{intro}

It is notoriously hard to unravel the nature of a finitely presented group $\G$. In order to do so, one must explore how the group can act on different kinds of objects. The most primitive objects to consider are  finite sets. Actions on finite sets capture only the finite images of groups, so the power of such actions to explain the nature of $\Gamma$ is limited by the answer to the fundamental question: to what extent is $\G$ determined by its set of finite quotients? This compelling question has re-emerged with different emphases throughout the history of group theory, and in recent years it has been animated by a rich interplay with geometry and low-dimensional topology, e.g. \cite{BF}, \cite{BCR}, \cite{BridR}, \cite{BRW}, \cite{Fu}, \cite{He}, \cite{Wilk} and \cite{WZ}.

The finite images of $\G$ are encoded in its profinite completion $\wh{\G}$ (see \S \ref{prelims_profinite}).  If $\G$ has elements that do not survive in any finite quotient, then one cannot hope to recover $\G$ by studying $\wh{\G}$, so it is natural to restrict attention to residually finite groups, i.e.~groups where every finite subset injects into some finite quotient. In its rawest form, the fundamental question then becomes: which finitely generated, residually finite groups $\G$ are {\em profinitely rigid in the absolute sense}, i.e.~have the property that if $\Lambda$ is finitely generated and residually finite, then $\wh{\Lambda}\cong \wh{\G}$ implies ${\Lambda}\cong {\G}$.

Variations on this problem have driven many advances in group theory in recent years but there has been very little progress towards identifying infinite groups that are profinitely rigid in the absolute sense. It is easy to see that finitely generated abelian groups are rigid, but beyond that one immediately struggles; virtually cyclic groups need not be rigid for example \cite{Baum74}. If $\G$ satisfies  a group law  -- for example if it is nilpotent or solvable -- then $\wh{\Lambda}\cong \wh{\G}$ will imply that $\Lambda$ satisfies the same law, provided it is residually  finite. In such cases, the question of absolute profinite rigidity reduces to a question of {\em relative} profinite rigidity, where one asks if $\wh{\G}$ distinguishes $\G$ from all other groups in a restricted class. If the relative context is sufficiently tame, this observation allows one to identify examples of profinitely rigid groups: for example, the free nilpotent group of fixed class on a fixed n
 umber of generators is profinitely rigid (although many other nilpotent groups are not, \cite{GS}).

The pursuit of relative profinite rigidity theorems has provided a focal point for a rich body of research in recent years, particularly in geometric contexts. This includes many settings in which the groups are {\em{full-sized}} in the sense that they contain non-abelian free subgroups. (Note that a full-sized group cannot satisfy a law.) For example, every Fuchsian group can be distinguished from any other lattice in a connected Lie group by its finite quotients \cite{BCR}. Such relative theorems do not lead to absolute profinite rigidity, however, because in the absence of a group law it is extremely difficult to rule out the possible existence of an utterly exotic  $\Lambda$, finitely generated and residually finite, with $\wh{\Lambda}\cong \wh{\G}$. Thus, despite our understanding of Fuchsian groups, it remains unknown whether finitely generated free groups are profinitely rigid in the absolute sense (which is an old question of Remeslennikov \cite[Question 5.48]{Rem}).

In this article we provide the first examples of full-sized groups that are profinitely rigid in the absolute sense. Our examples are fundamental groups of 3--dimensional hyperbolic orbifolds, i.e.~lattices in $\PSL(2,\C)$. Our main result is the following.

\begin{theorem}\label{t:main} 
There exist arithmetic lattices in $\PSL(2,\C)$ that are profinitely rigid in the absolute sense.  
\end{theorem}

These lattices are described in \S \ref{organisation} below, and in more detail in \S \ref{mainplayers}. We provide both uniform and non-uniform examples. For the moment we can construct only finitely many of each, but from these one can manufacture infinitely many commensurability classes of full-sized groups that are profinitely rigid. For instance, in \S \ref{full_sized} we prove that for one of our examples $\G_W$, the group $\G_W\times\Z^r$ is profinitely rigid for every $r\in\mathbb{N}$. Ultimately, one hopes to prove that all Kleinian groups are profinitely rigid, but such a result seems beyond the reach of current technology.

The lattices that we exhibit to prove Theorem \ref{t:main} have several important features in common, some arithmetic, some geometric, and some algebraic. A crucial algebraic characteristic is {\em representation rigidity}. The arguments that we develop around this idea are not specific to dimension $3$, are of a quite general nature, and apply especially in the setting of higher rank lattices. On the other hand, our geometric arguments make very strong use of the subgroup structure of Kleinian groups and $3$--manifold topology, and the corresponding parts of our proof do not extend to lattices in other settings. Indeed, any progress in this direction will require a far better understanding of the infinite index subgroups of general lattices than exists at present.  

We now outline how the strategy of our proof proceeds and describe how the key features of the examples enter the argument. First, representation rigidity: each of the lattices $\G$ that we study has very few irreducible representations into $\SL(2,\C)$ up to conjugacy, and all of the representations arise from the arithmetic structure of the lattice. The explanation for this is different in each case, as is the way in which it is exploited, but it is an indispensable feature of our arguments.  In particular it allows us to resolve the basic concern that an arbitrary finitely generated, residually finite group $\Lambda$ with $\widehat{\Lambda}\cong\widehat{\Gamma}$ may have no interesting representations to $\SL(2,\C)$.  It turns out that one can always construct a useful representation by working one prime at a time. In this way, we relate representations  $\G\to{\rm{SL}}(2,\C)$ to bounded representations  $\widehat{\G}\to{\rm{SL}}(2,\overline{\Q}_p)$. The restrictions of th
 e latter to $\Lambda < \widehat{\Lambda}=\widehat{\G}$ fit together to provide a Zariski-dense representation $\rho\colon\Lambda\to \SL(2,\C)$, and for the lattices that we consider,  one can use constraints coming from the arithmetic of $\G$  to force the image of $\rho$ to lie in $\G$, or perhaps a finite extension of $\G$. For more general lattices, it is difficult to control the images of representations constructed by such local techniques. (See the comments below concerning the $(2,3,7)$--triangle group.)

The arguments up to this point (\S \ref{s:rep}) are those that can be modified so as to apply quite generally; this theme is taken up in \cite{McSpit} and \cite{Spitler}. Thus, for example, one can prove that if  $\Lambda$ is a finitely generated group with the same finite quotients as ${\SL}(n,\Z)$, and $n\ge 3$, then there is a Zariski-dense representation $\Lambda\to {\SL}(n,\Z) < \SL(n,\R)$ that induces an isomorphism of profinite completions. 

Returning to our setting, it is the special way in which arithmetic and geometry dovetail in dimension $3$ that allows us to complete the proof. The arithmetic structure of a lattice $\Gamma<{\rm{PSL}}(2,\C)$ is encoded in its invariant trace-field $K\G = \Q({\rm{tr}}(\gamma^2) \colon \g\in\Gamma)$ and in the quaternion algebra $A\G$ over $K\G$. In all of our examples, $K\G$ is a number field of low degree ($2$ or $3$), $\G$ is closely related to  the unit group of  a maximal order in $A\G$, and the ramification of $A\G$ at finite places of $K\G$ is highly constrained. On the geometric side, in each case, we also exploit special features of the low index subgroups of $\G$, including the fact that the orbifold $\H^3/\G$ has a finite-sheeted covering of small index that is a surface bundle over the circle, explicit features of which are exploited heavily. Mostow Rigidity and volume calculations also play an important role. 

Once we have a Zariski-dense representation $\rho\colon\Lambda\to\G$ in hand, we argue, roughly speaking, that if $\rho$ were not surjective then $\Lambda$ would have a finite quotient that $\G$ does not have, contrary to hypothesis; and since $\wh{\Lambda}\cong\wh{\G}$ is Hopfian, the surjectivity of $\hat\rho\colon\wh{\Lambda} \to \wh{\G}$ implies injectivity, so $\rho$  is an isomorphism. In each case, arguments from classical 3--manifold topology reduce us quickly to the case where the image of $\rho$ has finite index in $\G$, and at that stage bespoke arguments about the low index subgroups of $\G$ and the finite linear quotients of $\G$ are invoked.

We wish to emphasize that our strategy for proving Theorem \ref{t:main} relies on the remarkable fact that the diverse array of arithmetic, geometric and algebraic features needed in the proof can be found in specific examples. It does not extend in an obvious way to any general classes of Kleinian groups. Looking further afield to the case of $\SL(n,\Z)$ with $n\geq 3$, this strategy reduces the issue of profinite rigidity to a question about finite quotients of finitely generated, infinite index, Zariski-dense subgroups of $\SL(n,\Z)$; but at this point we are blocked by how little is known about  the infinite index subgroups of $\SL(n,\Z)$ and the argument goes no further.

It is also interesting to consider how our argument breaks down if one replaces $\G$ by the triangle group $T = \Delta(2,3,7)$, which one also suspects might be profinitely rigid. One again has sufficient control to construct interesting representations of any group $\Lambda$ with $\widehat{\Lambda}\cong\widehat{T}$, but the arithmetic of the invariant trace-field of $T$ imposes less constraint than in the case of the 3--dimensional lattices we consider. Instead of concluding from $\wh{\Lambda}\cong\wh{T}$ that there exists a useful homomorphism $\Lambda\to T$, one has to contend with another possibility arising from the three real places of the invariant trace-field $K$. In this case, $\Lambda$ instead has a Zariski-dense representation to the Hilbert modular group $\PSL(2,R_K)$, about whose infinite index subgroups we again know embarrassingly little. There are, however, hyperbolic triangle groups where one has stricter control over their arithmetic, and in \cite{BMRS2} we 
 shall explain how the methods of the present paper can be extended to prove that certain of these co-compact Fuchsian groups are profinitely rigid in the absolute sense.

\subsection{Organisation of the paper}\label{organisation}

After gathering some basic facts about profinite completions in \S \ref{prelims_profinite} and trace-fields in \S \ref{s:algtraces}, in \S  \ref{s:rep}, we prove the general results that lead to the construction of $\rho\colon\Lambda\to{\rm{PSL}}(2,\C)$ as described above. Theorem \ref{T1} is the key result of a general nature in this section, and Corollary \ref{C1} applies it to the lattices $\G< {\rm{PSL}}(2,\C)$  that are the focus of subsequent sections. These lattices are presented in \S \ref{mainplayers}. Prominent among them is $\Gamma=\PSL(2,\Z[\omega])$, where $\omega$ denotes a primitive cube root of unity, i.e.~$\omega^2 + \omega + 1 = 0$. The orbifold $\H^3/\Gamma$ associated to this Bianchi group is a non-compact orientable arithmetic 3--orbifold, which is a 4--sheeted cover of the non-orientable orbifold of minimal volume. We also  describe all of the lattices that contain $\G$, as well as the Weeks manifold, which is the unique closed hyperbolic 3--manifold of minimal volume. In \S \ref{fewchars} we prove that, up to conjugacy, the only irreducible representations $\PSL(2,\Z[\omega]) \to{\rm{PSL}}(2,\C)$ are the inclusion and its complex conjugate; this is an instance of {\em Galois rigidity}, a concept that plays an important role in \S  \ref{s:rep}. Particular features of the arithmetic of $\Q(\omega)$, the invariant trace-field of $\PSL(2,\Z[\omega])$, also play a key role in \S \ref{fewchars}, as does an extension of Paoluzzi and Zimmermann's detailed analysis \cite{PZ} of epimorphisms $\G\to\PSL(2, \F)$, where $\F$ is a finite field. In \S \ref{s:rigid_bianchi} we prove that $\PSL(2,\Z[\omega])$ is profinitely rigid and in \S \ref{s:others} we prove that all of the lattices that contain it are also rigid. In \S \ref{s:weeks_rigidity} we turn our attention to uniform, torsion-free lattices and prove that the fundamental group of the Weeks manifold is profinitely rigid. 

Several of our proofs rely on computer calculations, implemented via Magma \cite{Bos}. The full details of each of these calculations can be found in the supplementary document \cite{magma_calcs}.

\bigskip


\noindent{\bf{Acknowledgements:}}~We want to acknowledge the great intellectual debt that we owe to Alexander Lubotzky. His seminal results and unparalleled vision underpin many aspects of the modern understanding of profinite completions of discrete groups; we thank him for these insights, as well as many helpful conversations.  We also thank Matt Stover for helpful conversations. During the period that this work was carried out, we benefited from the hospitality of many institutions, amongst which we mention The Hausdorff Institute for Mathematics (Bonn) which saw the final version of this work put in place. We thank all of them. We acknowledge with gratitude the financial support of the Royal Society (MRB), the National Science Foundation (DBM and AWR), and the Purdue Research Foundation (RS).

\section{Preliminaries concerning profinite groups}\label{prelims_profinite}

In this section we gather some results and remarks about profinite groups that we shall need. Let $\G$ be a finitely generated group.  The profinite completion of $\G$ is defined as $\wh{\G} = \ilim \G/N $ where the inverse limit is taken over the normal subgroups of finite index $N\ns\G$ ordered by reverse inclusion. $\wh{\G}$ is a compact, totally disconnected topological group.

The natural  homomorphism $i\colon\G\to\widehat\G$ is injective if and only if $\G$ is residually finite, and the image is always  dense. Hence, the restriction to $\G$ of any continuous epimomorphism from $\widehat{\Gamma}$ to a  finite group is onto.  A deep theorem of Nikolov and Segal \cite{NS} implies that if $\G$ is finitely generated then {\em every} homomorphism from $\widehat{\G}$ to a finite group is  continuous, and so every finite index subgroup of $\wh{\G}$ is open. The following proposition records the correspondence between subgroups of finite index in a discrete group and its profinite completion (see \cite[Prop 3.2.2]{RZ}). Note that we have used \cite{NS} to replace ``open subgroup'' in the profinite setting by ``finite index subgroup''. Given a subset $X$ of a profinite group $G$, we write $\overline X$ to denote the closure of $X$ in $G$.

\begin{proposition}\label{correspondence}
For every finitely generated, residually finite group $\Gamma$, there is a bijection from the set $\mathcal{X}$ of all finite index subgroups of $\Gamma$ to the set $\mathcal{Y}$ of all finite index subgroups of $\widehat{\Gamma}$.  If $\G$ is identified with its image in $\widehat{\Gamma}$, then this bijection takes $\Omega\in \mathcal{X}$ to $\overline{\Omega}$, while its inverse takes $\Lambda\in \mathcal{Y}$ to $\Lambda \cap \Gamma$. Also $[\Gamma : \Omega] = [\widehat{\Gamma}: \overline{\Omega}]$. Moreover, $\overline \Omega$ is normal in $\widehat\G$ if and only if $\Omega$ is normal in $\G$, in which case $\G/\Omega\cong\wh{\G}/\overline \Omega$.
\end{proposition}

\begin{corollary}\label{correspond_corollary}
If $\G_1$ and $\G_2$ are finitely generated groups with $\wh{\G_1}\cong\wh{\G_2}$, then there is a one-to-one correspondence between the subgroups of finite index in $\G_1$ and the subgroups of finite index in $\G_2$; this correspondence preserves index and takes normal subgroups to normal subgroups.
\end{corollary}

\begin{proof} 
Fixing an identification $\wh{\G_1}=\wh{\G_2}$, the correspondence is $\Omega \leftrightarrow \overline{\Omega}\cap\G_2$ for $\Omega<\G_1$.
\end{proof}

We will also find it useful to count conjugacy classes of finite index subgroups.

\begin{lemma}\label{l:conj-classes} 
Let $\G_1$ and $\G_2$ be finitely generated groups. If $\wh{\G_1}\cong\wh{\G_2}$ then, for every integer $d$, the number of conjugacy classes of subgroups of index $d$ in $\G_1$ is equal to the number in $\G_2$. Moreover, the sizes of the conjugacy classes are the same in $\G_1$ and $\G_2$.
\end{lemma}

\begin{proof} 
Each of $\G$ and $\wh{\G}$ acts on its set of index $d$ subgroups by conjugation and the bijection between these sets $H\mapsto \overline{H}$ is equivariant with respect to $\G\to\wh{\G}$. Thus the orbit structure for the action of $\G$ is an invariant of $\wh{\G}$.
\end{proof}

A standard argument shows that finitely generated groups $\Gamma_1,\Gamma_2$ have the same set of finite quotients if and only if $\wh{\G}_1\cong \wh{\G}_2$ (see \cite{DFPR}). We also require two other basic facts. First, if $\Epi(\G,Q)$ denotes the set of epimorphisms from the group $\G$ to the finite group $Q$, there is a bijection $\Epi(\widehat{\Gamma},Q)\to \Epi(\G,Q)$ (so if $\wh{\G}_1 \cong \wh{\G}_2$ then $\abs{\Epi(\G_1,Q)}=\abs{\Epi(\G_2,Q)}$). Second, we require the following elementary but  useful lemma. In the statement of the lemma, $H_1(\Gamma,\Z)$ denotes the first integral homology of $\Gamma$ and $b_1(\Gamma)$ denotes the first Betti number of $\Gamma$. 

\begin{lemma}\label{l:abel}
Let $\G_1$ and $\G_2$ be finitely generated groups. If $\G_1$ surjects a dense subgroup of $\wh\G_2$ then $b_1(\G_1)\geq b_1(\G_2)$. If $\wh\G_1\cong\wh\G_2$, then $H_1(\G_1,\Z)\cong H_1(\G_2,\Z)$.
\end{lemma}

Finally, we need to consider the relationship between the profinite completion of a group and the profinite completions of its subgroups. Let $\Gamma$ be a finitely generated, residually finite group with subgroup $\Delta < \Gamma$. The inclusion $\Delta\hookrightarrow\Gamma$ induces a continuous homomorphism $\wh{\Delta}\to\wh{\Gamma}$ whose image is $\-{\Delta}$. The map $\wh{\Delta}\to\wh{\Gamma}$ is injective if and only if $\-{\Delta}\cong\wh{\Delta}$; and we say that {\em $\Gamma$ induces the full profinite topology on $\Delta$} when that holds. As $\Gamma$ is finitely generated, this is equivalent to the statement that for every finite index subgroup $I<\Delta$ there is a finite index subgroup $S<\Gamma$ such that $S\cap\Delta \subset I$. Note that if $\Delta<\Gamma$ is of finite index, then $\Gamma$ induces the full profinite topology on $\Delta$. 

\section{Finitely generated subgroups of $\PSL(2,\C)$ with algebraic traces}\label{s:algtraces}

To fix notation, it will be convenient to record some basic facts about trace-fields of finitely generated subgroups of $\PSL(2,\C)$.  Kleinian groups are subgroups of $\PSL(2,\C)$, but throughout the paper it is often convenient to pass to $\SL(2,\C)$. Indeed many of the arguments are given in $\SL(2,\C)$ and are applied to the pre-image of a subgroup of $\PSL(2,\C)$. For concreteness, let $\phi\colon \mathrm{SL}(2,\mathbb{C}) \to \mathrm{PSL}(2,\mathbb{C})$ be the quotient homomorphism. For $\Delta$ a finitely generated subgroup of $\PSL(2,\mathbb{C})$, we define $\Delta_1 = \phi^{-1}(\Delta)$. It will be convenient to refer to $\Delta$ as being Zariski-dense in $\PSL(2,\C)$, by which we mean $\Delta_1$ is a Zariski-dense subgroup of $\SL(2,\C)$.

\subsection{Trace-fields}\label{ss:traces}

The \textit{trace-field} of $\Delta$ is defined to be $K_\Delta=\mathbb{Q}(\mathrm{tr}(\delta)~\colon~ \delta \in \Delta_1)$. We say that $\Delta$ has {\em integral traces} if $\tr(\delta) \in R_{K_\Delta}$ for all $\delta \in \Delta_1$ where $R_{K_\Delta}$ is the ring of algebraic integers in $K_\Delta$. The field $K_\Delta$ need not be a number field. However, one well-known situation when $K_\Delta$ is a number field is when $\Delta$ is \emph{rigid} (i.e. when $\Delta$ has only finitely many Zariski-dense representations to $(\mathrm{P})\SL(2,\C)$ up to conjugation). Let $\mathrm{X}_{\mathrm{zar}}(\Delta,\C)$ denote the set of Zariski-dense representations of $\Delta$ up to conjugacy in $(\P)\SL(2,\C)$. The following lemma is a consequence of \cite[Prop 6.6]{Rag}.

\begin{lemma} \label{trace_field_number_field}
If $\Delta<(\mathrm{P}) \SL(2,\C)$ is finitely generated and $\abs{\mathrm{X}_{\mathrm{zar}}(\Delta,\C)} < \infty$, then $[K_\Delta:\Q]<\infty$.
\end{lemma}

\subsection{Invariant trace-field and quaternion algebra}\label{invariant}

The \textit{invariant trace-field} of $\Delta$ is defined to be $K\Delta = \mathbb{Q}(\mathrm{tr}(\delta^2) ~\colon~ \delta \in \Delta_1)$. Alternatively, $K\Delta = K_{\Delta^{(2)}}$ where $\Delta^{(2)}$ is the subgroup of $\Delta$ generated by $\{\delta^2 ~\colon~ \delta \in \Delta\}$. As $K_\Delta/K\Delta$ is a multi-quadratic extension of degree $2^s$ for some $s \geq 0$, if $\Delta$ has no $\mathbb{Z}/2\mathbb{Z}$ quotient, then $s=0$ and $K_\Delta = K\Delta$. The group $\Delta_1$ generates a $K_\Delta$--quaternion algebra $A_0\Delta$ and $\Delta_1^{(2)} \leq \Delta_1$ generates a $K\Delta$--quaternion algebra $A\Delta$ called the \textit{invariant quaternion algebra}.  When $\Delta_1$ has integral traces, $\Delta_1$ generates an $R_{K_\Delta}$--order $\mathcal{O}\Delta$ in $A_0\Delta$ (see \cite[Ch 3]{MR} for more details on this material). Conversely, if $\Delta_1$ is contained in an order of $A_0\Delta$, then $\Delta_1$ has integral traces.  Similar statements hold fo
 r the invariant trace-field and quaternion algebra.

 \begin{remark} \label{Vinberg} 
An alternative description of the invariant trace-field was given by Vinberg \cite{Vin} who showed that $K\Delta = \mathbb{Q}(\mathrm{tr}(\mathrm{Ad}(\delta))~:~ \delta \in \Delta_1)$, where $\mathrm{Ad}\colon \mathrm{SL}(2,\mathbb{C}) \to \mathrm{Aut}(\mathfrak{sl}(2,\mathbb{C}))$ is the adjoint representation. 
 \end{remark}
 
\section{Linear Representations and Profinite Completions}\label{s:rep}

In this section we establish the representation theoretic results that will serve as the starting point for the proofs of our main result. The main result of this section is Theorem \ref{T1}, from which we isolate the special cases that will apply to the Kleinian groups described in the introduction; see Corollary \ref{C1} and Examples \ref{ex1}, \ref{ex2}. The material in this section can be adapted to other reductive algebraic groups (see \cite{McSpit}) but we will focus exclusively on ($\mathrm{P}$)${\rm{SL}}_2$, which is the setting for the rigidity results here and in \cite{BMRS2}.

Each of our proofs of profinite rigidity begins with an arithmetic Fuchsian or Kleinian group $i\colon\G\hookrightarrow$ ($\P$)$\SL(2,\C)$ whose only Zariski-dense representations, up to conjugacy, are the Galois conjugates of the inclusion (see \S 4.3). By way of arithmeticity, this {\em{``Galois rigidity"}} enables us to control the continuous representation theory of $\wh{\G}$ into ($\P$)$\SL(2, \overline{\Q}_p)$, for any finite rational prime $p$, and hence gain control on the representation theory of  any finitely generated group $\Delta$ with $\wh{\G}\cong\wh{\Delta}$. In particular, we obtain a bijection between the conjugacy classes of the Zariski-dense representations of $\Delta$ to $(\P)\SL(2,\C)$ and those of $\G$. Moreover, the arithmetic structure associated to these representations of $\Delta$  and $\G$ are closely related.  In the cases that are of particular interest to us, the specific control on the arithmetic data of $\G$ is sufficient to force $\Delta$ to have a Zariski-dense representation whose image lies in $\G$ (or a small index extension of it) and this representation is the key output from this section with regard to profinite rigidity.

\subsection{Preliminaries}\label{prelims}

We recall some standard terminology and fix notation. The set of primes $p \in \mathbb{N}$ together with $\infty$ will be denoted by $\mathrm{P}$. For each $p \in \mathrm{P}$, the metric completion of $\mathbb{Q}$ with respect to the $p$--adic metric will be denoted by $\mathbb{Q}_p$; note that $\mathbb{Q}_\infty=\mathbb{R}$. We will refer to the metric topology on $\mathbb{Q}_p$ and associated spaces (subsets, varieties over $\mathbb{Q}_p$, quotients, etc.) as the \emph{$p$--adic analytic topology}. For each $p \in \mathrm{P} \smallsetminus \set{\infty}$, the $p$-adic norm extends uniquely to the algebraic closure 
$\overline{\mathbb{Q}_p}$ of $\mathbb{Q}_p$ which is isomorphic to $\mathbb{C}$ as a field but not as a metric field: any isomorphism will induce maps on associated varieties which are continuous with respect to the Zariski topology but not the $p$--adic analytic topologies.
 
Given a number field $K/\mathbb{Q}$, we denote the degree of $K$ over $\mathbb{Q}$ by $n_K$ and the ring of algebraic integers of $K$ by $R_K$. For each $p \in \mathrm{P}$, the set of \emph{embeddings} (i.e.,~injective field homomorphisms) $\sigma\colon K \to \overline{\mathbb{Q}_p}$ is denoted by $E_K^p$. Note that $\abs{E_K^p} = n_K$. The set of embeddings of $K$ over all $p \in \mathrm{P}$ is denoted by $E_K$. The set of embeddings over all finite $p$ is denoted by $E_K^f$. For $\sigma \in E_K^\infty$,  we say that $\sigma$ is \emph{real} if $\sigma(K) \subset \mathbb{R}$ and that it is \emph{complex} otherwise. The group of continuous (analytic) automorphisms $\Aut_c(\overline{\mathbb{Q}_p})$ acts on $E_K^p$ by post-composition. The orbits are called the \emph{places of $K$ over $p$}. The set of places of $K$ over $p$ is denoted by $V_K^p$. The union of these sets over all $p \in \mathrm{P}\smallsetminus\{\infty\}$ is the set of \emph{finite} places $V_K^f$, and when the 
 set $V_K^\infty$ of infinite places is added we get the set $V_K$ of all places. Specific infinite places $v \in V_K^\infty$ can be termed \emph{real} or \emph{complex} as the action of $\Aut_c(\mathbb{C})\cong\Z/2$ preserves $\mathbb{R}$.  Given $v \in V_K$ associated to $\sigma\colon K \to \overline{\mathbb{Q}_p}$, we denote the metric closure of $\sigma(K)$ in $\overline{\mathbb{Q}_p}$ by $K_v$, noting that the isomorphism class of this locally compact field depends only on $v$.  

Given a quaternion algebra $B/K$ and $\sigma \in E_K^p$ with associated $v \in V_K^p$, we set $B_v = B \otimes_{\sigma(K)} K_v$. Note that $B_v$ is isomorphic to the topological closure of $B$ in $B \otimes_{\sigma(K)} \overline{\mathbb{Q}_p} \cong \mathrm{M}(2,\overline{\mathbb{Q}_p})$ in the $p$--adic analytic topology. We say that $B$ is \emph{ramified} at $v \in V_K$ if $B_v$ is a division algebra and we denote the set of ramified places  by $\mathrm{Ram}(B)$. We denote the subsets of finite and infinite ramified places of $B$ by $\mathrm{Ram}_f(B)$ and $\mathrm{Ram}_\infty(B)$. By class field theory, a quaternion algebra $B/K$ is determined up to $K$--isomorphism by the set of $K_v$--isomorphism classes $\set{B_v}_{v \in V_K}$; in particular, $B$ is determined by $\mathrm{Ram}(B)$, which is finite and of even cardinality (see for example \cite{MR}). We denote the group of units of $B$ by $B^\times$ and the group of reduced norm one elements by $B^1$. Given an $R_K$--orde
 r $\mathcal{O}$ of $B$ set $\mathcal{O}^\times = \mathcal{O} \cap B^\times$ and $\mathcal{O}^1 = \mathcal{O} \cap B^1$. Given $v \in V_K$ associated to $\sigma \in E_K$, we denote the closure of $\sigma(\mathcal{O})$ in $B_v$ by $\mathcal{O}_{v}$ and the closure of $\sigma(R_K)$ in $K_v$ by $R_{v}$. We note that $\mathcal{O}_{v}$ is an $R_{v}$--order in $B_v$. 

\subsection{Bounded Representations and Profinite Completions}

Given a finitely generated group $\Lambda$, a homomorphism $ \Lambda \to \mathrm{SL}(2,\overline{\mathbb{Q}_p})$ is said to be \emph{bounded} if its image is bounded (i.e.~pre-compact) in the $p$--adic analytic topology. For all finite primes $p$, the universal property of the profinite completion $\wh{\Lambda}$ provides a correspondence between representations $\wh{\Lambda}\to \mathrm{SL}(2,\overline{\mathbb{Q}_p})$ that are continuous in the $p$--adic analytic topology and bounded representations $\Lambda \to \mathrm{SL}(2,\overline{\mathbb{Q}_p})$; see Lemma \ref{L2}. Conjugacy classes are preserved by this correspondence, as is Zariski-denseness. Thus we obtain five related sets of representations, for which it is convenient to have names:
\begin{align*}
\mathrm{X}_b(\Lambda,\overline{\mathbb{Q}_p}) &= \set{\textrm{the conjugacy classes of bounded representations}~~\Lambda \to \mathrm{SL}(2,\overline{\mathbb{Q}_p})} \\ 
\mathrm{X}_c(\widehat{\Lambda},\overline{\mathbb{Q}_p}) &= \set{\textrm{the conjugacy classes of continuous representations}~~\widehat{\Lambda} \to \mathrm{SL}(2,\overline{\mathbb{Q}_p})}\\
\mathrm{X}_{\mathrm{zar}}(\Lambda,\overline{\mathbb{Q}_p}) &= \set{\textrm{the conjugacy classes of Zariski-dense representations}~~\Lambda \to \mathrm{SL}(2,\overline{\mathbb{Q}_p})} \\
\mathrm{X}_{c,\mathrm{zar}}(\widehat{\Lambda},\overline{\mathbb{Q}_p}) &= \set{\textrm{the conjugacy classes of Zariski-dense continuous representations}~~\widehat{\Lambda} \to \mathrm{SL}(2,\overline{\mathbb{Q}_p})}.
\end{align*}
The final set is $\mathrm{X}_{b,\mathrm{zar}}(\Lambda,\overline{\mathbb{Q}_p})= \mathrm{X}_b(\Lambda,\overline{\mathbb{Q}_p}) \cap \mathrm{X}_{\mathrm{zar}}(\Lambda,\overline{\mathbb{Q}_p})$. Given $p_1,p_2 \in \mathrm{P}$, any field isomorphism $\theta\colon \overline{\mathbb{Q}_{p_1}} \to \overline{\mathbb{Q}_{p_2}}$ induces an isomorphism of abstract groups $\mathrm{SL}(2,\overline{\mathbb{Q}_{p_1}}) \to \mathrm{SL}(2,\overline{\mathbb{Q}_{p_2}})$ and a bijection  $\theta_*\colon \mathrm{X}_{\mathrm{zar}}(\Lambda,\overline{\mathbb{Q}_{p_1}}) \longrightarrow \mathrm{X}_{\mathrm{zar}}(\Lambda,\overline{\mathbb{Q}_{p_2}})$,  so 
\begin{equation}\label{Eq:SillyEq}
\abs{\mathrm{X}_{\mathrm{zar}}(\Lambda,\overline{\mathbb{Q}_{p_1}})} = \abs{\mathrm{X}_{\mathrm{zar}}(\Lambda,\overline{\mathbb{Q}_{p_2}})}.
\end{equation}

\begin{proposition}\label{c:new} 
If $\Lambda$, $\Delta$ are finitely generated groups with $\wh{\Lambda}\cong\wh{\Delta}$ and $\abs{\mathrm{X}_{\mathrm{zar}}(\Lambda,\mathbb{C})}<\infty$, then $\abs{\mathrm{X}_{\mathrm{zar}}(\Delta,\mathbb{C})}=\abs{\mathrm{X}_{\mathrm{zar}}(\Lambda,\mathbb{C})}$.
\end{proposition} 

To prove this proposition, we need two lemmas.
 
\begin{lemma}\label{L2a}
Let $\Lambda$ be a finitely generated group. If $\phi_1,\dots,\phi_n\colon\Lambda\to \SL(2,\C)$ are Zariski-dense representations then for all but finitely many $p\in P\smallsetminus\{\infty\}$ there is a field isomorphism $\theta\colon\C\to \overline{\Q_p}$ such that each of the representations $\theta_*(\phi_i)$ is bounded. If $\mathrm{X}_{\mathrm{zar}}(\Lambda,\mathbb{C})$ is finite then $\abs{\mathrm{X}_{\mathrm{zar}}(\Lambda,\mathbb{C})} = \abs{\mathrm{X}_{b,\mathrm{zar}}(\Lambda,\overline{\mathbb{Q}_p})}$ and $\mathrm{X}_{\mathrm{zar}}(\Lambda,\overline{\mathbb{Q}_p}) = \mathrm{X}_{b,\mathrm{zar}}(\Lambda,\overline{\mathbb{Q}_p})$ for all but finitely many $p\in P\smallsetminus\{\infty\}$.
\end{lemma}
 
\begin{proof}
Let $R$ be the finitely generated subring of $\mathbb{C}$ generated by the matrix coefficients $\phi_i(\Lambda)$ for $i=1,\dots,n$. By Noether normalization, there exists $\alpha \in \mathbb{Z}$ such that $R[\alpha^{-1}]$ is isomorphic to a finite extension of $\mathbb{Z}[\alpha^{-1}][x_1,...,x_m]$. For any finite $p$ which does not divide $\alpha$, we obtain a ring embedding $\theta\colon R \to \overline{\mathbb{Q}_p}$ such that $\theta(R)$ is bounded by sending the transcendentals $x_1,...,x_m$ to $m$ independent transcendentals in $\mathbb{Z}_p$. This embedding has a unique extension to the field of fractions of $R$, which we identify with a subfield of $\mathbb{C}$. Using the axiom of choice, this in turn can be extended to the desired field isomorphism $\theta\colon \mathbb{C} \to \overline{\mathbb{Q}_p}$. If $\mathrm{X}_{\mathrm{zar}}(\Lambda,\mathbb{C})=\{\phi_1,\dots,\phi_n\}$, then $\theta_*$ defines an injection 
$\mathrm{X}_{\mathrm{zar}}(\Lambda,\mathbb{C}) \hookrightarrow \mathrm{X}_{b,\mathrm{zar}}(\Lambda,\overline{\mathbb{Q}_p})$ for all $p$ not dividing $\alpha$, and  \eqref{Eq:SillyEq} completes the proof.
\end{proof}
 
When $\wh{\G}\cong\wh{\Lambda}$, one can indirectly relate the complex representation theory of $\G$ to that of $\Lambda$ via $\mathrm{X}_c(\widehat{\Lambda},\overline{\mathbb{Q}_p})$.

\begin{lemma}\label{L2}
If $\Lambda$ is finitely generated group, then for each finite $p$, the map $ \mathrm{X}_c(\widehat{\Lambda},\overline{\mathbb{Q}_p}) \to \mathrm{X}_b(\Lambda,\overline{\mathbb{Q}_p})$ given by composing representations with the canonical map $\Lambda\to\wh{\Lambda}$ is a bijection, and it restricts to a bijection from $\mathrm{X}_{c,\mathrm{zar}}(\widehat{\Lambda},\overline{\mathbb{Q}_p})$ to $\mathrm{X}_{b,\mathrm{zar}}(\Lambda,\overline{\mathbb{Q}_p})$. 
\end{lemma}

\begin{proof} 
The image of any continuous representation $\widehat{\phi}\colon \widehat{\Lambda} \to \mathrm{SL}(2,\overline{\mathbb{Q}_p})$ is compact, so its restriction to the image of $\Lambda$ is bounded. Conversely, given $\phi\colon \Lambda \to \mathrm{SL}(2,\overline{\mathbb{Q}_p})$, since $\phi(\Lambda)$ is finitely generated, there exists a finite extension $F/\mathbb{Q}_p$ such that $\phi(\Lambda) < \mathrm{SL}(2,F)$. The closed subgroup $\mathrm{SL}(2,F) < \mathrm{SL}(2,\overline{\mathbb{Q}_p})$ is locally profinite (i.e.~Hausdorff, locally compact, and totally disconnected), so if $\phi$ is bounded then the topological closure $\overline{\phi(\Lambda)}$ of $\phi(\Lambda)$ is a profinite group, and hence $\phi$ extends to a continuous representation of $\wh{\Lambda}$. As conjugation in $\mathrm{SL}(2,\overline{\mathbb{Q}_p})$ is continuous and the image of $\Lambda$ is dense in $\wh{\Lambda}$, the correspondence between continuous representations of $\wh{\Lambda}$ and bounded r
 epresentations of $\Lambda$ induces a bijection $\mathrm{X}_c(\widehat{\Lambda},\overline{\mathbb{Q}_p}) \to \mathrm{X}_b(\Lambda,\overline{\mathbb{Q}_p})$. The correspondence preserves Zariski-denseness because the $p$--adic analytic topology is finer than the Zariski topology.
\end{proof} 

\begin{proof}[Proof of Proposition \ref{c:new}] 
If $\abs{\mathrm{X}_{\mathrm{zar}}(\Delta,\mathbb{C})}$ were greater than $N(\Lambda):=\abs{\mathrm{X}_{\mathrm{zar}}(\Lambda,\mathbb{C})}$, then Lemma \ref{L2a}  would imply for almost all finite primes $\abs{\mathrm{X}_{b,\mathrm{zar}}(\Delta,\overline{\mathbb{Q}_p})}>N(\Lambda)$. By Lemma \ref{L2},
\[ \abs{\mathrm{X}_{b,\mathrm{zar}}(\Delta,\overline{\mathbb{Q}_p})} = \abs{\mathrm{X}_{c,\mathrm{zar}}(\wh{\Delta},\overline{\mathbb{Q}_p})} = \abs{\mathrm{X}_{c,\mathrm{zar}}(\wh{\Lambda},\overline{\mathbb{Q}_p})} = \abs{\mathrm{X}_{b,\mathrm{zar}}(\Lambda,\overline{\mathbb{Q}_p})}. \] 
However, 
$N(\Lambda)=\abs{\mathrm{X}_{b,\mathrm{zar}}(\Lambda,\overline{\mathbb{Q}_p})}$ for almost all $p$ by Lemma \ref{L2a}, which is a contradiction.
\end{proof}

\subsection{Galois Rigidity}\label{s:galois-rigid} 

In order to produce the necessary representations needed for profinite rigidity, we require lattices with the fewest possible Zariski-dense representations. Consider the situation where $\G$ is a finitely generated group with $\rho\colon \G\rightarrow (\P)\SL(2,\C)$ a Zariski-dense representation such that $K=K_{\rho(\G)}$ is a number field of degree $n_K$. As $K=\Q(\theta)$ for some algebraic number $\theta$, the Galois conjugates $\theta=\theta_1,\dots,\theta_{n_K}$ of $\theta$ provide embeddings $\sigma_i\colon K\to\C$ defined by $\theta\mapsto\theta_i$.  These in turn can be used to build $n_K$ Zariski-dense non-conjugate representations $\rho_{\sigma_i}\colon \G \to (\P)\SL(2,\C)$ with the property that $\tr(\rho_{\sigma_i}(\gamma))=\sigma_i(\tr\rho(\gamma))$ for all $\gamma\in \G$. We will refer to these as {\em Galois conjugate representations}. We see from this construction that $|\mathrm{X}_{\mathrm{zar}}(\G,\mathbb{C})|\geq  n_K$ which motivates the following defini
 tion. 

\begin{definition}[Galois Rigid]\label{def:galois-rigid}
Let $\G$ be a finitely generated group and let $\rho\colon\G\to (\P)\SL(2,\C)$ be a Zariski-dense representation whose trace-field $K_{\rho(\G)}$ is a number field. If $|\mathrm{X}_{\mathrm{zar}}(\G,\mathbb{C})|= n_{K_{\rho(\G)}}$, we say that $\G$ is {\em Galois rigid} (with associated field $K_{\rho(\G)}$).
\end{definition}

When we say that a subgroup of $(\P)\SL(2,\C)$ is Galois rigid, we implicitly take $\rho$ to be the inclusion map. Note that if $\G$ is Galois rigid, then any irreducible representation with infinite image can serve as $\rho$ as all such representations are Galois conjugate. In particular, $K_{\rho(\G)}$ is an intrinsic invariant of $\G$, as is the  quaternion algebra $A_0\G:=A_0\rho(\G)$ and  the group homomorphism 
$\G \to \rho(\G) \hookrightarrow A_0\G^1$. Integrality of traces plays a key role in what follows and is a necessary
property beyond Galois rigidity. From the discussion in \S \ref{invariant} integrality is assured if $\G$ is contained in an order $\mathcal{O}<A_0\G$.

\begin{lemma}\label{l:integral-traces} 
If $\G$ is a finitely generated, residually finite group with the property that for each Zariski-dense representation $\rho\colon\G\to(\P)\SL(2,\C)$ we have $\mathrm{tr}(\rho(\gamma)) \in R_{K_{\rho(\Gamma)}}$ for all $\gamma \in \Gamma$ and for some number field $K_{\rho(\Gamma)}$, then $\mathrm{X}_{\mathrm{zar}}(\Gamma,\overline{\mathbb{Q}_p}) =  \mathrm{X}_{b,\mathrm{zar}}(\G,\overline{\mathbb{Q}_p})$ for all primes $p$. 
\end{lemma}

\begin{proof} 
For convenience we work with $\rho\colon\G\to \SL(2,\C)$, a Zariski-dense representation as in the statement of lemma. We replace $\mathbb{C}$ with $\overline{\Q_p}$ by fixing a field isomorphism $\C\to\overline{\Q_p}$; note that this preserves algebraic integers. Then $K=K_{\rho(\G)}$ is a number field, and the traces of elements $\rho(\gamma)$ are algebraic integers in $K\subset \overline{\Q_p}$.  In particular, for any place $v\in V_K^p$, $K_v$ is a finite extension of $\Q_p$, and integrality implies that $\tr(\rho(\gamma))\in R_v$ for all $\gamma \in \G$.  In particular, since $R_v$ is compact, $\tr(\rho(\gamma))$ is bounded. Thus the representation $\rho\colon \G\to{\rm{SL}}(2,\overline{\Q_p})$ has bounded traces.  But it is a standard argument that a representation with bounded traces is bounded. Briefly, following the construction of \S \ref{invariant}, taking the $R_v$ span of $\rho(\G)$ in $\mathrm{M}(2,\overline{\Q_p})$ determines an $R_v$--order of a quaternion alg
 ebra $B_v$.  The elements of norm $1$ in this order form a compact subgroup of $\rm{SL}(2,\overline{\Q_p})$ (see \cite[Ch 7]{MR} for example), and so the representation is bounded.
\end{proof}

Combining \eqref{Eq:SillyEq} with Lemmas \ref{L2}, \ref{l:integral-traces} and Proposition \ref{c:new} we have:
 
\begin{lemma}\label{L:SimRig}
Let $\Lambda$ be a Galois rigid group with associated field $K$, each of whose Zariski-dense representations to $(\P)\SL(2,\C)$ has integral traces. Suppose that $\Delta$ is a finitely generated group with $\widehat{\Delta} \cong \widehat{\Lambda}$. Then,
\begin{itemize}
\item[(i)]
$\abs{\mathrm{X}_{b,\mathrm{zar}}(\Lambda,\overline{\mathbb{Q}_p})} = n_K$ for all finite $p \in \mathrm{P}$,
\item[(ii)]
$\abs{\mathrm{X}_{\mathrm{zar}}(\Delta,\mathbb{C})} = n_K$,
\item[(iii)]
$\mathrm{X}_{b,\mathrm{zar}}(\Delta,\overline{\mathbb{Q}_p}) = \mathrm{X}_{\mathrm{zar}}(\Delta,\overline{\mathbb{Q}_p})$ for all finite $p$.
\end{itemize}
\end{lemma}

\begin{remark} \label{galois_FA}
If $\G$ is a finitely generated group with \emph{Serre's Property FA} (i.e.~$\G$ cannot act on a tree without a global fixed point), then the $(\P)\SL(2,\C)$--character variety $\mathrm{X}(\G,\C)$ is finite (see \cite{BZ}). In particular, any group with Property FA  has the property that the set  $\mathrm{X}_{\mathrm{zar}}(\Gamma,\C) \subset \mathrm{X}(\Gamma,\C)$ is finite. However,  Property FA and Galois rigidity are distinct forms of rigidity. For instance, the triangle group $\Delta(6,6,6)$ has Property FA but is not Galois rigid. On the other hand, there are rational homology 3--manifolds  $\Sigma$ such that $\Sigma$ is hyperbolic and $\pi_1\Sigma$ is Galois rigid, but $\pi_1\Sigma$ does not have property FA because $\Sigma$ is Haken;  the result of $(10,1)$--Dehn surgery on the knot $5_2$ is such a manifold. There are also hyperbolic 3--manifolds $M$ with $b_1(M) >0$  whose fundamental group is Galois rigid, for example the manifold $M_6$ described in \S \ref{s:weeks_rigidity} below.
\end{remark}

\subsection{Profinite rigidity via Galois rigidity}\label{profinite_from_galois}

We now fix a number field $K$, a quaternion algebra $B/K$, and a maximal order $\mathcal{O} < B$. We will assume that $\Gamma < \mathcal{O}^1$ is a finitely generated subgroup such that $K_\G= K$ and we identify $A_0\G$ with $B$. We denote the inclusion $\Gamma \to B^1$ by $\phi$. We now state the main technical result of this section.

\begin{theorem}\label{T1}
Let $K,B,\mathcal{O}$, and $\Gamma$ be as above, assume that $\Gamma$ is Galois rigid and assume that $\Delta$ is a finitely generated residually finite group with $\widehat{\Delta} \cong \widehat{\Gamma}$. Then there is a number field $K'$, a quaternion algebra $B'/K'$, a maximal order $\mathcal{O}' < B'$, and a Zariski-dense homomorphism $\phi'\colon \Delta \to (\mathcal{O}')^1 < (B')^1 \subset \SL(2,\C)$ such that the following conditions hold:
\begin{itemize}
\item[(i)]
$\Delta$ is Galois rigid with associated field $K'$.
\item[(ii)]
There are bijective functions $\tilde{\tau}\colon E_{K'} \to E_K$ and $\tau\colon V_{K'} \to V_{K}$ with $K'_w \cong K_{\tau(w)}$ for all $w \in V_{K'}$. Hence, $K$ and $K'$ are arithmetically equivalent and have isomorphic adele rings. 
\item[(iii)]
$B'_w \cong B_{\tau(w)}$ for all $w \in V_{K'}^f$. 
\item[(iv)]
Up to isomorphism, there are only finitely many possibilities for $K'$, $B'$, and $\mathcal{O}'$.
\end{itemize}
\end{theorem}

\noindent An extension of Theorem \ref{T1} covering more general algebraic groups is developed in the Ph.D thesis of the fourth author \cite{Spitler}.

\begin{proof}[Proof of Theorem \ref{T1}]
By Lemma \ref{L:SimRig} (ii), $\mathrm{X}_{\mathrm{zar}}(\Delta,\mathbb{C})$ is nonempty. Let $\psi \in \mathrm{X}_{\mathrm{zar}}(\Delta,\mathbb{C})$, set $K' =K_{\psi(\Delta)}$, let $B'/K'$ be the quaternion algebra $A_0\psi(\Delta)$ and let $\phi'\colon \Delta \to (B')^1$ denote the homomorphism given by viewing $B'$ as an abstract quaternion algebra. By Lemma \ref{L:SimRig} (ii), we have $\abs{\mathrm{X}_{\mathrm{zar}}(\Delta,\mathbb{C})} = n_K$, so $n_{K'}\leq n_K$, and we conclude that $K'$ is a number field. To show that $\Delta$ is Galois rigid we need to prove that $n_{K'} = n_K$. This will be done by establishing that $K$ and $K'$ are arithmetically equivalent, from which $n_K = n_{K'}$  follows (see \cite{Per}).   Indeed we will prove that $\mathbb{A}_K$ and $\mathbb{A}_{K'}$, the adele rings of $K$ and $K'$, are isomorphic. 

The first part of this is to construct bijective functions $\tau_p\colon V_{K'}^p \to V_K^p$ for each $p \in \mathrm{P}\setminus \{\infty\}$ such that $K_{v'}' \cong K_{\tau_p(v')}$ for each $v' \in V_{K'}^p$. To that end, using the homomorphism $\phi'\colon \Delta \to (B')^1$ and $p \in \mathrm{P}\setminus \{\infty\}$, we obtain $n_{K'}$ non-conjugate, Zariski-dense representations  $\{\phi'_{\sigma'_1},\dots,\phi'_{\sigma'_{n_{K'}}}\}$ into $\mathrm{SL}(2,\overline{\mathbb{Q}_p})$ as follows.  For $\sigma' \in E_{K'}^p$ associated to $w \in V_{K'}^p$, we have an algebra injection $B' \to B'\otimes_{\sigma'(K')} K'_w \subset \mathrm{M}(2,\overline{\mathbb{Q}_p})$ which induces an injective homomorphism when restricted to $\phi'(\Delta)$ and hence a homomorphism $\phi'_{\sigma'}\colon \Delta \to \mathrm{SL}(2,\overline{\mathbb{Q}_p})$ with Zariski dense image.  Hence $\phi'_{\sigma'} \in \mathrm{X}_{\mathrm{zar}}(\Delta,\overline{\mathbb{Q}_p})$. 
For distinct $\sigma'_1,\sigma'_2 \in E_{K'}^p$, the representations $\phi'_{\sigma'_1},\phi'_{\sigma'_2}$ are not conjugate as they have distinct characters. 

Since $\wh{\G}\cong\wh{\Delta}$,  from Lemma \ref{L:SimRig} (i) we have $\abs{\mathrm{X}_{b,\mathrm{zar}}(\G,\overline{\mathbb{Q}_p})} = n_K$, and by Lemma \ref{L:SimRig} (iii), $\mathrm{X}_{b,\mathrm{zar}}(\Delta,\overline{\mathbb{Q}_p}) = \mathrm{X}_{\mathrm{zar}}(\Delta,\overline{\mathbb{Q}_p})$ for all finite $p$. 

Given this set up, for $\sigma' \in E_{K'}^p$, we can define functions $\tilde{\tau}_p : E_{K'}^p \to E_K^p$ as follows: as noted above, the representation $\phi'_{\sigma'}\colon \Delta \to \mathrm{SL}(2,\overline{\mathbb{Q}_p})$ is Zariski-dense with bounded image; using $\wh{\G}\cong \wh{\Delta}$, we deduce from Lemma \ref{L2} that there is a unique $\sigma \in E_K^p$ so that $\widehat{\phi'_{\sigma'}}$ and $\widehat{\phi_\sigma}$ are conjugate representations, and so we set $\tilde{\tau}_p(\sigma') = \sigma$. Note that $\tilde{\tau}_p\colon E_{K'}^p \to E_K^p$ is injective since the representations induced by distinct embeddings of $K'$ are not conjugate. 

To show that the functions $\tilde{\tau}_p$ give rise to injective functions $\tau_p\colon V_{K'}^p \to V_K^p$, we must show that if $\sigma'_1,\sigma'_2 \in E_{K'}^p$ are $\mathrm{Aut}_c(\overline{\mathbb{Q}_p})$--equivalent, then $\tilde{\tau}_p(\sigma'_1),\tilde{\tau}_p(\sigma'_2) \in E_K^p$ are $\mathrm{Aut}_c(\overline{\mathbb{Q}_p})$--equivalent. Let $\sigma'_1,\sigma'_2 \in E_{K'}^p$ with $\sigma'_2 = \varphi \circ \sigma'_1$ for some $\varphi \in \mathrm{Aut}_c(\overline{\mathbb{Q}_p})$ and set $\sigma_i = \tilde{\tau}_p(\sigma'_i)$ for $i=1,2$. By definition of $\tilde{\tau}_p$, we have continuous Zariski-dense representations 
\[ \widehat{\phi'_{\sigma'_1}},~\widehat{\phi'_{\sigma'_2}},~\widehat{\phi_{\sigma_1}},~\widehat{\phi_{\sigma_2}}\colon \widehat{\Gamma}\longrightarrow \mathrm{SL}(2,\overline{\mathbb{Q}_p}) \]
such that $\widehat{\phi'_{\sigma'_i}}$ is conjugate to $\widehat{\phi_{\sigma_i}}$ for $i=1,2$ and $\widehat{\phi'_{\sigma'_2}} = \varphi_* \circ \widehat{\phi'_{\sigma'_1}}$ where $\varphi_*\colon \mathrm{SL}(2,\overline{\mathbb{Q}_p}) \to  \mathrm{SL}(2,\overline{\mathbb{Q}_p})$ is the automorphism induced by $\varphi$. So up to conjugation, $\widehat{\phi_{\sigma_2}} = \varphi_* \circ \widehat{\phi_{\sigma_1}}$, and restricting to $\Gamma$ gives $\phi_{\sigma_2} = \varphi_* \circ \phi_{\sigma_1}$ up to conjugation, so $\sigma_2 = \varphi \circ \sigma_1$. Thus, $\tilde{\tau}_p$ induces an injective function $\tau_p\colon V_{K'}^p \longrightarrow V_K^p$. Varying $p \in \mathrm{P}\setminus \set{\infty}$ induces injective functions $\tilde{\tau}_f\colon E_{K'}^f \to E_K^f$ and  $\tau_f\colon V_{K'}^f \to V_K^f$. 

Since $\Delta$ and $\Gamma$ are dense in $\widehat{\Delta}\cong\widehat{\Gamma}$, for each finite embedding $\sigma' \in E_{K'}^f$, the sets $\set{\mathrm{tr}~\phi'_{\sigma'}(\delta)}_{\delta \in \Delta}$ and $\set{\mathrm{tr}~\phi_{\tilde{\tau}(\sigma')}(\gamma)}_{\gamma \in \Gamma}$ are dense in 
\[ \set{\mathrm{tr}~\widehat{\phi'_{\sigma'}}(x)}_{x \in \widehat{\Delta}} = \set{\mathrm{tr}~\widehat{\phi_{\tilde{\tau}(\sigma')}}(x)}_{x \in \widehat{\Gamma}}. \] 
Hence, the fields generated by $\set{\mathrm{tr}~\phi'_{\sigma'}(\delta)}_{\delta \in \Delta}$ and $\set{\mathrm{tr}~\phi_{\tilde{\tau}(\sigma')}(\gamma)}_{\gamma \in \Gamma}$ have the same closures and so $K'_w \cong K_{\tau_f(w)}$ where $w \in V_{K'}^f$ is associated to $\sigma'$. To establish that $\tau_f$ is a bijection we need the following lemma. This appears to be well-known, but we could not find a proof in the literature, so we give it in the Appendix.

\begin{lemma}\label{L4}
If $K$, $K'$ are number fields and $\tau_f\colon V_{K'}^f \to V_K^f$ is an injective map with $K'_w \cong K_{\tau_f(w)}$ for all $w \in V_{K'}^f$, then $\tau_f$ is a bijection. 
\end{lemma}

Given this, $\tau_f$ is a bijection, and hence $\tilde{\tau}_p,\tilde{\tau}_f,\tau_p$, and $\tau_f$ are all bijections. That $\tilde{\tau}_f, \tau_f$ can be extended to bijective functions $\tilde{\tau}\colon E_{K'} \to E_K$ and $\tau\colon V_{K'} \to V_K$ follows from the fact that $\tau_f$ being a bijection implies that $K$ and $K'$ are arithmetically equivalent (see \cite{Per}), and hence $K$ and $K'$ have the same number of real and complex embeddings/places. Hence, $n_K = n_{K'}$. This establishes (i).

Now Iwasawa \cite{Iwasawa} proved that the existence of $\tau$ is equivalent to the fields $K$ and $K'$ having isomorphic adele rings. This establishes (ii). 

To prove (iii), we argue as follows. The algebras $B$ and $B'$ are generated over $K$ and $K'$ by $\phi(\Gamma)$ and $\phi'(\Delta)$ respectively.  Using the bijection $\tau$ established in (ii), these quaternion algebras are also generated over $K$ and $K'$ by $\phi_{\tilde{\tau}(\sigma')}(\Gamma)$ and $\phi'_{\sigma'}(\Delta)$ respectively (up to isomorphism). Moreover, if we now take the algebras generated by $\phi_{\tilde{\tau}(\sigma')}(\Gamma)$ and $\phi'_{\sigma'}(\Delta)$ over $ K_{\tau(w)} = K'_w$ we obtain quaternion algebras $B_{0,\tau(w)}$ and $B'_{0,w}$ that are isomorphic to $B_{\tau(w)}$ and $B'_w$ respectively. Now these quaternion algebras are also generated over $K_{\tau(w)} = K'_w$ by $\widehat{\phi_{\tilde{\tau}(\sigma')}}(\widehat{\Gamma})$ and $\widehat{\phi'_{\sigma'}}(\widehat{\Delta})$. By the construction in the proof of (i), these groups are conjugate, and  so the quaternion algebras $B_{0,\tau(w)}$ and $B'_{0,w}$ are isomorphic.  Hence $B_{\tau(w)}
 $ and $B'_w$ are isomorphic as required.

It remains to show that $\phi'$ has image contained in some maximal order $\mathcal{O}'$ of $B'$. As noted in \S \ref{invariant}, it suffices to show $\set{\mathrm{tr}~\phi'(\delta)}_{\delta \in \Delta} \subset R_{K'}$  To that end, let $R_{w'}$ be the associated local ring for $w' \in V_{K'}^f$. As the representations $\phi'_{\sigma'}$ are bounded for each $\sigma' \in E_{K'}^p$ and each finite $p$, we have 
\[ \set{\mathrm{tr}~\phi'_{\sigma'}(\delta)}_{\delta \in \Delta} \subset R_{w'} \] 
for each $w \in V_{K'}^f$. Since $\set{\mathrm{tr}~\phi'_{\sigma'}(\delta)}_{\delta \in \Delta} \subset K'$ for all $\sigma' \in E_{K'}^f$ and
\[ R_{K'} = \bigcap_{w' \in V_{K'}^f} (K' \cap R_{w'}), \]     
we have $\set{\mathrm{tr}~\phi'(\delta)}_{\delta \in \Delta} \subset R_{K'}$ as needed. Thus, $\phi'(\Delta)$ generates an order over $R_{K'}$ which is contained in some maximal order $\mathcal{O}'$. 

For (iv), we must analyze how $K'$, $B'$, and $\mathcal{O}'_{B'}$ can fail to be uniquely determined. Since $K$ and $K'$ are arithmetically equivalent, they have the same Galois closure, and so there are a finite number of possibilities for $K'$. For each possible $K'$, the quaternion algebra $B'/K'$ is determined by the following information:
\begin{itemize}
\item[(1)]
For each $v \in \mathrm{Ram}_f(B)$, a choice of $w \in \mathrm{Ram}_f(B')$ with $K_v \cong K'_w$. 
\item[(2)]
With (1), we get a bijective function $\mathrm{Ram}_f(B') \to \mathrm{Ram}_f(B)$ and this determines $\mathrm{Ram}_f(B')$. It only remains to determine the possible choices for $\mathrm{Ram}_\infty(B')$. If $\abs{\mathrm{Ram}_f(B)}$ is even, then we can take $S=\mathrm{Ram}_\infty(B')$ for any set $S$ of real places with $\abs{S}$ even. If $\abs{\mathrm{Ram}_f(B)}$ is odd, then we can take $S=\mathrm{Ram}_\infty(B')$ for any set $S$ of real places with $\abs{S}$ odd. 
\end{itemize}
As $\mathrm{Ram}(B')$ is finite, there are only finitely many possibilities for $B'$. Finally, fixing $B'$, there are only finitely many maximal orders $\mathcal{O}'_{B'}$ up to isomorphism. Hence, we can then take our list of possible codomain groups for $\phi'$ to be $(\mathcal{O}'_{B'})^1$ where $K'$, $B'$, and $\mathcal{O}'_{B'}$ each range over all of the above choices. 
\end{proof}

We record a specific corollary that will be utilized in our proofs (a variation of which is also used in \cite{BMRS2}).

\begin{definition}[Locally uniform]
We say that a quaternion algebra $B/K$ is \emph{locally uniform} if for each $v,v' \in V_K^f$ with $K_v \cong K_{v'}$ we have $B_v \cong B_{v'}$.
\end{definition}

\begin{corollary}\label{C1}
Let $\Gamma,\Delta$, $K,K'$, $B$, $B'$ be as in Theorem \ref{T1}. 
\begin{itemize}
\item[(i)]
If $K$ is Galois or has exactly one complex place, then $K' \cong K$. 
\item[(ii)]
If $K' \cong K$, $K$ has at most one real place, and $B$ is locally uniform, then $B' \cong B$.
\item[(iii)]
\begin{itemize}
\item[(a)]
If $K$ is imaginary quadratic and $\mathrm{Ram}(B)=\emptyset$, then $K \cong K'$ and $B \cong B'$.
\item[(b)]
If $n_K=3$, $K$ has one real place, and $\mathrm{Ram}(B) = \set{v_1,v_2}$ where $v_1$ is a real place and $v_2 \in V_K^p$ for a finite $p$ with $V_K^p = \set{v_2}$, then $K \cong K'$ and $B \cong B'$.
\end{itemize}
\item[(iv)]
If $K,B$ satisfy (ii), (iii) (a) or (iii) (b), then there exists a Zariski-dense representation $\phi'\colon \Delta \to \mathcal{O}_B^1$ for some maximal order $\mathcal{O}_B < B$.
\end{itemize}
\end{corollary}

\begin{proof}
For (i), by Theorem \ref{T1} (ii), $K,K'$ are arithmetically equivalent, and so have isomorphic Galois closures by \cite[Thm 1]{Per}. In particular, when $K$ is Galois, then $K \cong K'$. If $K$ has exactly one complex place, then $K \cong K'$ by \cite[Cor 1.2]{CHLR}. 

For (ii), we now assume that $K \cong K'$ and view $B'$ as a quaternion algebra over $K$. We also view $\tau_f$ as a bijective function $\tau_f\colon V_K^f \to V_K^f$ with  $K_v \cong K_{\tau(v)}$ and $B'_v \cong B_{\tau(v)}$ for each $v \in V_K^f$. To prove that $B,B'$ are isomorphic, it suffices to prove that $B_v \cong B_v'$ for all $v \in V_K$. For $v \in V_K^f$, set $w = \tau_f^{-1}(v)$. As $K_v \cong K_w$ and $B$ is locally uniform, $B_v \cong B_w$. By Theorem \ref{T1}, we have $B_w \cong B_v'$ and so $B_v \cong B_v'$. It remains to prove $B_v \cong B_v'$ when $v \in V_K^\infty$. At each complex place $v$, we have $B_v \cong B_v' \cong \mathrm{M}(2,\mathbb{C})$. As $K$ has at most one real place, whether or not $\mathrm{Ram}(B)$, $\mathrm{Ram}(B')$ contains the real place (if it exists) depends only on the even/odd parity of $\abs{\mathrm{Ram}_f(B)}$, $\abs{\mathrm{Ram}_f(B')}$. As $\mathrm{Ram}_f(B) = \mathrm{Ram}_f(B')$ by the above, we see that either both $B,B'$ ram
 ify at the real place or both $B,B'$ split over the real place. In particular, $B_v \cong B_v'$ when $v$ is real. This proves that $B_v \cong B_v'$ for all $v \in V_K$ as needed. 

Parts (a) and (b) of (iii) follow from (i) and (ii). Part (iv) follows from (a) and (b) of (iii) and Theorem \ref{T1}.
\end{proof}

\noindent We single out two particular cases of Corollary \ref{C1} that will be important in what follows.

\begin{example}\label{ex1} 
In the setting of Corollary \ref{C1}, suppose that  $\G$ is a non-uniform lattice in $\PSL(2,\C)$ with $K_\Gamma=\Q(\sqrt{-3})$.  In this case $B\cong \mathrm{M}(2,\Q(\sqrt{-3}))$, and selecting $\mathcal{O} = \mathrm{M}(2,\Z[\omega])$ where $\omega^2+\omega+1=0$,  Corollary \ref{C1}(iii)(a) applies. By Corollary \ref{C1}(iv), if $\Delta$ is a finitely generated, residually finite group with $\wh{\Delta}\cong \wh{\G}$, then there is a maximal order $\mathcal{L} < \mathrm{M}(2,\Q(\sqrt{-3}))$ and we have $\phi':\Delta \to \mathcal{L}^1$. In this case, since the class number of $\Q(\sqrt{-3})$ is $1$, $\mathrm{M}(2,\Q(\sqrt{-3}))$ has type number $1$: there is a unique conjugacy class of maximal order in $\mathrm{M}(2,\Q(\sqrt{-3}))$ (see \cite[Ch 7.6]{MR}). Thus we may conjugate so that both $\G$ and $\phi'(\Delta)$ are contained as subgroups of $(\P)\SL(2,\Z[\omega])$. 
\end{example} 

\begin{example}\label{ex2} 
Assume $\G$ is as in Corollary \ref{C1}, that $K_\Gamma=\Q(\theta)$, that $\theta^3-\theta^2+1=0$, and that $B$ is ramified at both the real place of $K_\Gamma$ and the unique place  $\nu$ of norm $5$. Then Corollary \ref{C1}(iii)(b) applies. As in Example \ref{ex1}, we can deduce from  Corollary \ref{C1}(iv) that for any finitely generated, residually finite group with $\wh{\Delta}\cong \wh{\G}$ there is a maximal order $\mathcal{O}_B < B$ with $\Delta \to \mathcal{O}_B^1$. As in Example \ref{ex1}, the class number of $K_\Gamma$ is $1$, so $B$ has type number $1$, and hence we can conjugate so that $\G$ and the image of $\Delta$ are both contained as subgroups of $\mathcal{O}_B^1$.
\end{example}

\subsection{Congruence quotients}\label{congruence_quotients} 

The results and proofs contained in this section allow us to make some additional conclusions about congruence quotients that will be useful later. Suppose that $\G$ is as in the statement of Theorem \ref{T1}, and suppose also that $\G$ has the following property: {\em Every finite quotient of $\G$ the form $(\P)\SL(2,\mathbb{F}_{p^n})$ arises as a quotient of $\widehat{\phi_\sigma}(\widehat{\Gamma})$ for some $\sigma$; i.e.~as a congruence quotient.} Then this property passes to $\Delta$, in the following sense. 

Theorem \ref{T1} gives a Zariski-dense representation $\phi'$ of $\Delta$ and the proof of that theorem exhibits a family of finite quotients of the form $(\P)\SL(2,\mathbb{F}_{p^n})$ 
by restriction of $\widehat{\phi'_{\tilde{\tau}^{-1}(\sigma)}}$ to $\Delta \subset \widehat{\Delta}$. Moreover, since $\wh{\Delta}\cong \wh{\G}$, all such quotients of $\Delta$ (and $\phi'(\Delta)$) will arise as a quotient of $\widehat{\phi'_{\sigma'}}(\widehat{\Delta})$ for some  $\sigma' \in E_{K'}^p$, i.e. as a congruence quotient.
 
This construction can be made more explicit in the setting of Corollary \ref{C1}, where we have a Zariski-dense representation $\phi'\colon \Delta \to \mathcal{O}^1$ for some maximal $R_K$--order $\mathcal{O} \subset B$ (with $K=K_\G=K_{\phi'(\Delta)}$).  For convenience
assume that $\phi'(\Delta)=L < \G$. Let $v\in V_K^f$ with local ring $R_v$ having a local uniformizer $\pi_v$.  Assume that $B$ is unramified at $v$.  In such a case, the finite quotients of the form $(\P)\SL(2,\mathbb{F}_{p^n})$ come from restricting the following sequence of homomorphisms to $\G$ and $L$:
\[ \mathcal{O}^1\hookrightarrow \SL(2,R_v)\rightarrow \SL(2,R_v/\pi_vR_v) \cong \SL(2,\mathbb{F}_{p^n}). \] 
If we denote the composition of these homomorphisms by $\eta_v$, and restrict to $\G$ and to $L$, then $\ker\eta_{v}$ determines subgroups $\G(v)=\G\cap \ker\eta_v$ and $L(v)=L\cap \ker\eta_v$. Our arguments show that $\G/\G(v) = L/L(v)$.

\section{The main players}\label{mainplayers}

In this section we describe some of the important features of the groups that we shall prove are profinitely rigid.

\subsection{The group $\G=\PSL(2,\Z[\omega])$ and other non-uniform lattices of small covolume in ${\rm{Isom}}(\H^3)$}\label{s:small}

Consider a regular ideal tetrahedron $\tau$ in $\H^3$; this is unique up to isometry and we write $v_0$ to denote its volume. Let $c_0$ be the centre of the inscribed sphere in $\tau$. This sphere touches each 2--dimensional face $f$ at the centre $c(f)$ of the inscribed circle in $f$ and this circle touches each edge $e$ in the boundary of $f$ at a point we denote $c(e)$. Let $\xi$ be an ideal vertex of $e$ and consider the ideal tetrahedron $T_0$ with vertices $(\xi,c(e),c(f),c_0)$. Note that $\tau$ decomposes into $24$ copies of $T_0$ corresponding to the different choices of $(\xi,e,f)$; thus $T_0$ has volume $v_0/24$. The dihedral angles of $T_0$ can be calculated by observing the number of translates of $T_0$ around each edge; the angles at the ideal vertices are $\pi/6, \pi/2, \pi/3$ and at the opposite edges they are $\pi/3, \pi/2, \pi/2$. Thus, in the notation of \cite{BLW}, $T_0=T[3,2,2; 6,2,3]$. We write $\Lambda_0$ to denote the subgroup of ${\rm{Isom}}(\H^3)$ gen
 erated by reflections in the hyperplanes containing the faces of $T_0$. This is (up to conjugacy) the unique non-uniform lattice of smallest co-volume in ${\rm{Isom}}(\H^3)$ -- see \cite{Mey}; it has $T_0$ as fundamental domain, so its co-volume is $v_0/24$. Below is a presentation for $\Lambda_0$ that we will use later:
\[ \Lambda_0 = \innp{ x,y,z,w~\mid~ x^2 = y^2 = z^3 = w^2 = (xy)^3 = (xz)^2 = (xw)^2 = (yz)^3 = (yw)^2 = (zw)^6 = 1}. \]
The index $2$ subgroup $\G_0<\Lambda_0$ consisting of orientation-preserving isometries is $\PGL(2,\Z[\omega])$, and from \cite{BLW} we have the presentation
\[ \G_0 = \innp{ x,y,z ~\mid~ x^3= y^2= z^2 = (yx^{-1})^3 = (zx^{-1})^2 = (yz)^6 =1}. \] 
$T_0$ has one face in the boundary of $\tau$ and the union of $T_0$ with its reflection in this face is the tetrahedron $T_1=T[3,2,2; 3,3,3]$. We write $\Lambda_1$ to denote the subgroup of ${\rm{Isom}}(\H^3)$ generated by reflections in the hyperplanes containing the faces of $T_1$; it has covolume $v_0/12$ and is an index 2 subgroup of $\Lambda_0$. The following is the natural (Coxeter) presentation:
\[ \Lambda_1 = \innp{ x,y,z,w ~\mid~ x^2 = y^2 = z^2 = w^2 = (xy)^2 = (xz)^2 = (xw)^3 = (yz)^3 = (yw)^3 = (zw)^3 =1}. \]
The index 2 subgroup of orientation preserving isometries in $\Lambda_1$ is $\G=\PSL(2,\Z[\omega])$; thus $\G$ has covolume $v_0/6$ and as in \cite{BLW} we have the presentation
\begin{equation}\label{present}
\G = \innp{ x,y,z ~\mid~ x^3 = y^2 = z^2 = (yx^{-1})^3 = (zx^{-1})^3 = (yz)^3 = 1}.
\end{equation}

\begin{lemma}\label{l:abelG}
The abelianizations of $\G, \G_0, \Lambda_0, \Lambda_1$ are, respectively, $\Z/3\Z, \Z/2\Z, (\Z/2)^2$ and $\Z/2\Z$.
\end{lemma}

\begin{proof} 
The abelianizations of $\G,\, \G_0$ and $\Lambda_1$ are readily calculated from the above presentations. As $\G_0$ has index 2 in $\Lambda_0$, the abelianization of $\Lambda_0$ has order at most 4. And since $\Lambda_0$ has a further subgroup of index 2, namely $\Lambda_1$, we have $H_1(\Lambda_0,\Z) = (\Z/2\Z)^2$.
\end{proof}

Since $H_1(\Lambda_0,\Z) = (\Z/2\Z)^2$, there is a third subgroup of index $2$ in $\Lambda_0$ besides $\G_0$ and $\Lambda_1$. We denote this subgroup by $\Lambda_2$, and note that since $\Gamma$ is the commutator subgroup of $\Lambda_0$, we have $\Gamma<\Lambda_2$. We will make use of the following presentation for $\Lambda_2$ in the proofs of the main results: 
\[ \Lambda_2 = \innp{x,y,z ~\mid~ x^2 = y^2= z^3 = (zx)^3 = (xy)^6 = zyz^{-1}y = 1}. \]
Note that the abelianization of $\Lambda_2$ is isomorphic to $\Z/6\Z$.

We also include a diagram of the above subgroups for the reader's convenience
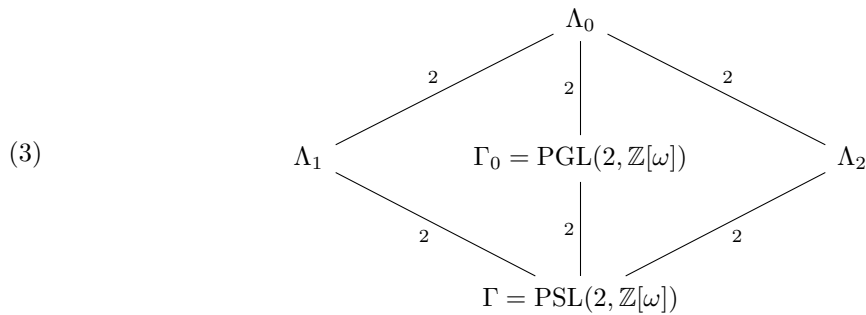
\begin{equation}\label{Eq:Natural}
\begin{tikzcd}
& & \Lambda_0 \arrow[rrdd,"2",dash] \arrow[lldd,"2"',dash] \arrow[dd,"2"',dash] & & \\ & & & & \\ \Lambda_1 \arrow[rrdd,"2"', dash] & & \Gamma_0 = \PGL(2,\Z[\omega]) \arrow[dd,"2"', dash] &  &\Lambda_2 \arrow[lldd,"2", dash] \\ & & & & \\ & & \Gamma = \PSL(2,\Z[\omega]) & &
\end{tikzcd}
\end{equation}

\subsection{The five lattices $\G, \G_0, \Lambda_0$, and $\Lambda_1$, $\Lambda_2$}

From consideration of the smallest $3$--dimensional non-compact hyperbolic orbifold, we have identified four lattices containing $\G$, namely $\G_0$, $\Lambda_0$, $\Lambda_1$ and $\Lambda_2$. 

\begin{lemma}\label{l:upLattice} 
The only lattices in $\rm{Isom}(\H^3)$ that contain $\G$ are $\G$, $\G_0$, $\Lambda_0$, $\Lambda_1$, and $\Lambda_2$.
\end{lemma}

\begin{proof} 
$\Lambda_0$ is the unique non-uniform lattice of minimal co-volume $v_0/24$ and all other lattices commensurable with $\Lambda_0$ have co-volume at least $v_0/12$ (see \cite{Ad} and \cite{Mey}). It follows that $\G$, which has co-volume $v_0/6$, must have index $2$ in any lattice $\Delta\neq \Lambda_0$ that properly contains it. This forces $\Gamma$ to be normal in $\Delta$. By Mostow Rigidity, the normaliser of $\G$ is itself a lattice; and since it contains $\Lambda_0$ it must be equal to $\Lambda_0$. Thus $\Delta$ is a subgroup
of index 2 in $\Lambda_0$.   
\end{proof}

\subsection{The orbifolds $\H^3/\G$ and $\H^3/\G_0$}

The quotient orbifolds $\H^3/\G_0$ and $\H^3/\G$ are shown in Figure 1. Note that each of these orbifolds has a single cusp, with cusp cross-section homeomorphic to a Euclidean $2$--orbifold which is a $2$--sphere with three cone points. In the case of $\H^3/\Gamma_0$, this orbifold is $S^2(2,3,6)$, which is $S^2$ with three cone points with cone angles $\pi$, $2\pi/3$ and $\pi/3$. In the case of  $\H^3/\Gamma$ it is $S^2(3,3,3)$; i.e.~the cone angles are all $2\pi/3$.

\begin{figure}[h]
\centering

\begin{tikzpicture} [scale=.5]
  \coordinate (O) at (0,0);
  \coordinate (A) at (0,3.4);
  \coordinate (B) at (0,1.2);
  \coordinate (C) at (-1.6,-1.4);
  \coordinate (D) at (1.6,-1.4);
  \coordinate (E) at (-3.2,-2.4);
  \coordinate (F) at (3.2,-2.4);
  
  \filldraw (A) circle (1pt);
  \filldraw (B) circle (1pt);
  \filldraw (C) circle (1pt);
  \filldraw (D) circle (1pt);
  \filldraw (E) circle (1pt);
  \filldraw (F) circle (1pt);
  
  \draw (A) to [edge label = $3$] (B);
  \draw (B) to [edge label' = $2$] (C);
  \draw (B) to [edge label = $3$] (D);
  \draw (C) to [edge label' = $2$] (D);
  \draw (C) to [edge label' = $3$] (E);
  \draw (D) to [edge label = $3$] (F);

  \draw [densely dashed] (-4,0) arc [start angle = 180, end angle = 0,
    x radius = 40mm, y radius = 8mm];
  \draw [white, line width = 5pt](-4,0) arc [start angle=-180, end angle = 0,
    x radius = 40mm, y radius = 8mm];
  \draw (-4,0) arc [start angle=-180, end angle = 0,
    x radius = 40mm, y radius = 8mm];
    
  \draw (O) circle [radius=4cm];
  
  \draw node at (0,-5) {$\mathbb{H}^3/\Gamma$};

 \coordinate (O) at (10,0);
  \coordinate (A) at (10,3.4);
  \coordinate (B) at (10,1.2);
  \coordinate (C) at (8.4,-1.4);
  \coordinate (D) at (11.6,-1.4);
  \coordinate (E) at (6.8,-2.4);
  \coordinate (F) at (13.2,-2.4);
  
  \filldraw (A) circle (1pt);
  \filldraw (B) circle (1pt);
  \filldraw (C) circle (1pt);
  \filldraw (D) circle (1pt);
  \filldraw (E) circle (1pt);
  \filldraw (F) circle (1pt);
  
  \draw (A) to [edge label = $6$] (B);
  \draw (B) to [edge label' = $2$] (C);
  \draw (B) to [edge label = $2$] (D);
  \draw (C) to [edge label' = $3$] (D);
  \draw (C) to [edge label' = $3$] (E);
  \draw (D) to [edge label = $2$] (F);

  \draw [densely dashed] (6,0) arc [start angle = 180, end angle = 0,
    x radius = 40mm, y radius = 8mm];
  \draw [white, line width = 5pt](6,0) arc [start angle=-180, end angle = 0,
    x radius = 40mm, y radius = 8mm];
  \draw (6,0) arc [start angle=-180, end angle = 0,
    x radius = 40mm, y radius = 8mm];
    
  \draw (O) circle [radius=4cm];
  
  \draw node at (10,-5) {$\mathbb{H}^3/\Gamma_0$};
\end{tikzpicture}
\caption{}
\end{figure}
These cusps are {\em{rigid}} in the sense that no non-trivial Dehn surgery can be performed on them (see \cite{DuMe}). This cusp-rigidity is a crucial feature of $\G$ and $\G_0$. Both groups have Serre's property FA and $\G$ is the unique Bianchi group with this property; see \cite{Fi}. As noted earlier (see Remark \ref{galois_FA}), it follows that the $\PSL(2,\C)$--character varieties of these groups, $\mathrm{X}(\G,\C)$ and $\mathrm{X}(\G_0,\C)$ are both finite. Unfortunately, our proof requires that $\G$ be Galois rigid. We shall deduce this from the tight control that one can get on the representations of $\G$ and $\G_0$ over finite fields; this is the subject of \S \ref{fewchars}.

\subsection{The Weeks manifold}\label{weeks}

Let  $M_W={\Bbb H}^3/\Gamma_W$ denote the Weeks manifold; this is the unique closed orientable hyperbolic 3--manifold of minimal volume.  It can be obtained by performing  $(-5,1)$ surgery on one component of the Whitehead link and $(5,2)$ surgery on the other.  Using SnapPy \cite{CDW}, a presentation for $\G_W$ can be computed and we record it for future use:
\[ \innp{a,b~|~ababa^{-1}b^2a^{-1}b,abab^{-1}a^2b^{-1}ab}. \]
From this, we see that $H_1(M_W,{\Bbb Z})\cong {\Bbb Z}/5{\Bbb Z}\times {\Bbb Z}/5{\Bbb Z}$.

We shall need some facts about the arithmetic structure of the Weeks manifold $M_W$ (see \cite[Ch 4.8.3, Ch 9.8.2]{MR} for more details). The invariant trace-field of $\G_W$ is $K=K\Gamma_W={\Bbb Q}(\theta)$ where $\theta^3-\theta^2+1=0$ is as in Example \ref{ex2} (a field of discriminant $-23$). In addition, the invariant quaternion algebra $B_W$ coincides with the algebra $B$ of Example \ref{ex2}. As noted in Example \ref{ex2},  there is a unique conjugacy class of maximal orders in $B_W$ and for a choice of maximal order $\mathcal{O}\subset B_W$, the group $\Gamma_W$ is a normal subgroup of index $3$ in the group of units $\G_\mathcal{O}^1$. As shown in \cite{MedV}, the orbifold $\H^3/\G_\mathcal{O}^1$ can be described as $(3,0)$ Dehn surgery on $5_2$ (i.e.~$M_W$ is the $3$--fold cyclic branched cover of the knot $5_2$). Furthermore, the minimal orbifold in the commensurability class of $\G_W$ arises from a maximal arithmetic group $\Gamma_{\mathcal{O}}$ which is the normalizer of $\G_\mathcal{O}^1$ in $\PSL(2,\C)$, and $\Gamma_{\mathcal{O}}/\G_\mathcal{O}^1\cong \Z/2\Z\times \Z/2\Z$. Since $M_W$ is non-Haken, $\mathrm{X}(\G_W,\C)$ consists of a finite number of points \cite{BZ}.  By \cite[Cor 6.3]{ReWan}, $\G_W$ has very few irreducible representations into $\PSL(2,\C)$.

\begin{proposition}\label{reidwang}
In $\mathrm{X}(\G_W,\C)$, there only three characters of irreducible representations, namely the characters of the faithful discrete representation, its complex conjugate, and a $\PSU(2)$--representation arising from the real ramified place of $B_W$.
\end{proposition}

\begin{corollary}\label{weeks_galois_rigid}
$\Gamma_W$ is Galois rigid.
\end{corollary}

\section{$\PSL(2,\Z[\omega])$ has very few characters}\label{fewchars}

In this section we prove the following result, which establishes that $\PSL(2,\Z[\omega])$ is Galois rigid.

\begin{theorem}\label{2chars} 
$\G=\PSL(2,\Z[\omega])$ has two characters of Zariski-dense representations in $\PSL(2,\C)$, namely the characters associated to the inclusion homomorphism and its complex conjugate.  In particular, $\G$ is Galois rigid.
\end{theorem}

For brevity, throughout this section $\G$ will always denote $\PSL(2,\Z[\omega])$.
\subsection{$\PSL(2,\F)$--quotients}\label{Paoluzzi_Zimmermann}

We begin by describing $\PSL(2,{\F})$ quotients where $\F$ is a finite field. We shall see that homomorphisms from $\G$ onto $\PSL(2,\F)$, with $\F$ a finite field, arise exclusively from the arithmetic of $\Z[\omega]$. This result is due in large part to L. Paoluzzi and B. Zimmermann \cite[Thm 6.3]{PZ}. First we need some notation. Given a prime $p\in \Z$, the ideal $p\Z[\omega]$ is a prime ideal  of $\Z[\omega]$ if $p= -1~\mod~6$ or $p=2$, and splits as a product of two distinct prime ideals $\mathcal{P}_1\mathcal{P}_2$ if
$p = 1~\mod~6$. If $p=3$ then $3\Z[\omega] = \<\sqrt{-3}\>^2$, the square of a prime ideal. In the first of these three cases, the residue class field $\Z[\omega]/\mathcal{P}$ is a field with $p^2$ elements. In the remaining cases, for each prime $\mathcal{P}$ that arises, the field $\Z[\omega]/\mathcal{P}$ has $p$ elements. Each of the ring homomorphisms $\Z[\omega]\rightarrow \Z[\omega]/\mathcal{P}$ induces a {\em reduction homomorphism} $\pi_{\mathcal{P}}\colon\PSL(2,\Z[\omega])\rightarrow \PSL(2,\Z[\omega]/\mathcal{P})$. These homomorphisms are onto, and therefore pick out the following collection of finite quotients of $\Gamma$:

\begin{enumerate}
\item $\G\onto\PSL(2,{\Bbb F}_{p^2})$ when $p=2$ or $p=-1~\mod~6$;
\item a pair of quotient maps $\G\onto\PSL(2,{\F}_p)$ when $p=1~\mod~6$;
\item a single quotient map $\G\onto\PSL(2,{\F}_p)$ when $p=3$.
\end{enumerate}

\noindent We will denote this collection of finite quotients of $\G$ by ${\bf P}$. When the meaning is clear, we shall say $\PSL(2,{\F}_p)\in {\bf P}$, but more formally the elements of ${\bf P}$ are the normal subgroups $\Gamma(\mathcal{P}) < \G$ that arise as the kernels of the reduction homomorphisms $\pi_{\mathcal{P}}$. (Thus we identify two surjections if they differ by composition with an automorphism of the target.) The kernels $\Gamma(\mathcal{P})$ are torsion-free apart from when $\mathcal{P}=\<\sqrt{-3}\>$, in which case $\Gamma(\mathcal{P})$ has elements of order $3$, for example the image in $\PSL(2,\C)$ of the element $\begin{pmatrix} \omega & 0\cr 0& \omega^2\cr\end{pmatrix}$. It is shown in \cite{Alp} that there is a unique normal subgroup of index $12$ in $\G$, and this coincides with $\G(\<\sqrt{-3}\>)$. (Although \cite{Alp} mistakenly labels $\G$ as $\PSL(2,\Z[\sqrt{-3}])$, the results in the paper actually refer to $\G$.)

The results of \cite{PZ} (in particular Theorem 6.3) classify {\em admissible} epimorphisms  of $\G$ onto groups of the form $\PSL(2,\F)$ with $\F$ a finite field, where admissible means that the kernel is assumed to be torsion-free. In the applications that we require, torsion cannot be avoided, and we therefore require the following lemma to augment  \cite{PZ}.  

\begin{lemma}\label{torsion_normal}
Let $\Omega$ be a proper normal subgroup of  $\G$ and suppose that $\Omega$ contains a non-trivial element of finite order. Then $\Omega$ has index $3$ or $12$ in $\G$: in the first case, $\Omega$ is the commutator subgroup of $\G$, and in the second case $\Omega=\G(\<\sqrt{-3}\>)$.
\end{lemma}

\begin{proof} 
By the Cartan fixed-point theorem, every finite subgroup of $\G$ fixes a point in $\H^3$ and is therefore conjugate into one of the vertex-stabilizer groups for the orbifold description of $\H^3/\G$ given in \S \ref{s:small}. In terms of the presentation (\ref{present}) of $\G$, these stabilizers are 
\[ \Omega_1= \innp{y,z~\mid~ y^2=z^2=(yz)^3=1} \]
which is isomorphic to $S_3$, and two copies of the alternating group $A_4$, namely 
\begin{align*}
\Omega_2 &=\innp{x,z\mid z^2=x^3=(zx^{-1})^3=1} \\ 
\Omega_3 &= \innp{x,y\mid y^2=x^3=(yx^{-1})^3=1}. 
\end{align*}
Thus if a normal subgroup $\Omega$ contains a non-trivial element of finite order, then it intersects at least one of $\Omega_1$, $\Omega_2$ or $\Omega_3$ non-trivially.

Setting $x=1$ trivializes $\G$, so $N$ cannot contain $\Omega_2$ or $\Omega_3$. The only proper normal subgroup of $A_4$ is the commutator subgroup, which has order $4$. The commutator subgroup $\Omega_2'$ contains $z$ while the commutator subgroup $\Omega_3'$ contains $y$, and setting either $y=1$ or $z=1$ in $\G$ reduces $\G$ to its abelian quotient $\Z/3\Z$. So either $\Omega\cap \Omega_i=\Omega_i'$ for $i=2,3$, in which case $\Omega=[\G,\G]$, or else $\Omega$ intersects both $\Omega_2$ and $\Omega_3$ trivially. In the latter case, $\Omega$ contains neither $y$ nor $z$.

If $\Omega$ intersects $\Omega_2$ and $\Omega_3$ trivially, then it must intersect $\Omega_1\cong S_3$ in a proper normal subgroup, and the only such is  $\<yz\>$. A direct calculation shows that setting $yz=1$ in $(\ref{present})$ yields $A_4$ as a quotient. Thus, in this case, either $\Omega$ has index $12$ in $\G$ (and then by \cite{Alp} $\Omega=\G(\<\sqrt{-3}\>)$), or else $\G/\Omega$ is the unique quotient $\Z/3\Z$ of $A_4$. But this last possibility would contradict the assumption $\Omega_2\cap \Omega=1$, because $y$ has order $2$.
\end{proof}

\begin{theorem}\label{controllingmodp}
Let $\F$ be a finite field and let $\phi\colon\G \rightarrow \PSL(2,\F)$ an epimorphism. Then $\PSL(2,\F)\in{\bf P}$ and $\phi$ is a reduction homomorphism $\pi_{\mathcal{P}}$.
\end{theorem}

\begin{proof}
Paoluzzi and Zimmermann \cite[Thm 6.3]{PZ} proved this theorem for epimorphisms with torsion-free kernel, and Lemma \ref{torsion_normal} removes the need to assume the kernel is torsion-free.
\end{proof}

\begin{remark} 
A description of admissible homomorphisms for $\PGL(2,\Z[\omega])$ is also given in \cite{PZ}, however we will not need to appeal to that here.
\end{remark}

\subsection{Some comments on Strong Approximation}\label{strongapprox}

The proof of Theorem \ref{2chars} requires the Strong Approximation Theorem \cite{Wei}. Suppose that $\rho\colon \G \to \mathrm{PSL}(2,\mathbb{C})$ is a representation with Zariski-dense image $\Delta$.  Since $\G^{\ab}\cong \Z/3\Z$, $\Delta$ cannot have any $\Z/2\Z$ quotients and so we deduce that $K_\Delta=K\Delta$, and $A_0\Delta=A\Delta$ (recall \S \ref{s:algtraces}).  As $\G$ is rigid (see Remark \ref{galois_FA}), we have $[K_\Delta:\Q]<\infty$ by Lemma \ref{trace_field_number_field}. Hence, $A\Delta^1$ is an absolutely almost simple, simply connected algebraic group defined over the number field $K\Delta$.  The following is a consequence of \cite[Thm 8.1]{Wei} stated in a form that is useful for us (Remark \ref{Vinberg} is pertinent here).

\begin{theorem} [Weisfeiler]\label{RepSec:T1}
In the notation above, for all but a finite number of $K\Delta$--primes $\mathcal{P}$ with residue class field $\F_{\mathcal{P}}$, there is a reduction homomorphism $\Delta\to \PSL(2,\F_{\mathcal{P}})$ which is onto.\end{theorem}

\subsection{Proof of Theorem \ref{2chars}}
Suppose that $\rho\colon \G\to \PSL(2,\C)$ is a Zariski-dense representation. Since $\G$ has Property FA, there are only a finite number of  Zariski-dense representations up to conjugacy, and they all have  integral traces \cite[Ch 1, \S 6]{Serre}. In particular, if $\Delta=\rho(\G)$ then $\tr(\Delta) \subset R_{K\Delta}$. By  Theorem \ref{RepSec:T1} we get for all but a finite number of prime ideals $\mathcal{P}$ of $R_{K\Delta}$, an epimorphism $\G\colon \Delta\to \PSL(2,{\F}_q)$ where $q$ is the cardinality of the residue class field $\F_q=R_{K\Delta}/\mathcal{P}$.  By Theorem \ref{controllingmodp}, for all but a finite number of primes $\mathcal{P}$ the finite groups $\PSL(2,{\F}_q)$ correspond to those in $\bf P$.  Hence for all but a finite number of rational primes $p$, the field $K\Delta$ and $\Q(\omega)$ have the same splitting type for prime ideals.  Since $[\Q(\omega):\Q]=2$, it follows from \cite{Per} that $K\Delta=\Q(\omega)$. 

The algebra $A\Delta$ is defined over $\Q(\omega)$, and $\mathcal{O}=\mathcal{O}\Delta \subset A\Delta$ is an order since $\Delta$ has integral traces.  We shall prove that $A\Delta=\mathrm{M}(2,\Q(\omega))$.  Once this is established, we complete the argument as follows.  As noted in Example \ref{ex1}, there is only one type of maximal order in $\mathrm{M}(2,\Q(\omega))$, and so we can conjugate so that $\Delta<\G$. Now $\G$ is residually $\PSL(2,{\F})$ (using the proof of residual finiteness), and so if $g\in\ker(\G\to \Delta)$ then we can find a prime ideal $\mathcal{P}$ of $\Z[\omega]$ so that the image of $g$ in the composition $\G\to \Delta\to \PSL(2,\Z[\omega]/\mathcal{P})$ is not trivial, so cannot lie in the kernel. Hence $\G\to \Delta$ is injective, and so by Mostow Rigidity $\Delta=\G$ and $\rho:\G\rightarrow \Delta$ is the inclusion map up to conjugation in $\Isom(\H^3)$. Thus the character of $\rho$ is
that of the inclusion map or its complex conjugate.

We now prove that $A\Delta\cong \mathrm{M}(2,\Q(\omega))$.  Suppose that this is not the case, then by the classification theorem for quaternion algebras, we can find at least two distinct primes $\mathcal{P}_1$ and $\mathcal{P}_2$ with $A\Delta$ ramified at these primes. Assume that the residue class fields $\Z[\omega]/\mathcal{P}_i$ have characteristic $p_i$ for $i=1,2$. Let $A_i=A\Delta\otimes_{\Q(\omega)} \Q(\omega)_{\mathcal{P}_i}$ for $i=1,2$ and $\mathcal{O}_i$ the unique maximal order in these division algebras (see \cite[Ch 6.4]{MR} for these and the following details). Now $\mathcal{O}_i$ admits a filtration that restricts to $\mathcal{O}_i^1$ to provide a filtration of the following form $\mathcal{O}_i^1 > G^i_1 > G^i_2 \ldots $, where each of these subgroups is normal in $\mathcal{O}_i^1$, $\cap G_j^i=1$, $\mathcal{O}_i^1/G_1^i$ is cyclic of order dividing $p_i^4-1$ and the $G_j^i/G_{j+1}^i$ are abelian $p_i$--groups. Note that for $i=1,2$, and all $j\geq 1$, the 
 groups $G_j^i$ are are pro--$p_i$ groups whose intersection is the identity subgroup. Hence any element of finite order in $G_j^i$ has order a power of $p_i$ for $i=1,2$. Identifying $\Delta$ with its image in $\mathcal{O}_i^1$ under the inclusion map, first note that for $i=1,2$, $\Delta$  cannot be a subgroup of $G_1^i$. The reason is as follows. From \S \ref{s:small}, we see that $\G$ is generated by elements of orders $2$ and $3$, hence $\Delta$ is generated by elements of orders $2$ and $3$. Indeed from the proof of Lemma \ref{torsion_normal}, $\rho(x)$, $\rho(y)$ and $\rho(z)$ are all non-trivial elements of order $2$ (in the case of $\rho(y)$ and $\rho(z)$) or $3$ (in the case of $\rho(x)$). However, from the previous paragraph, the groups $G_1^i$ cannot contain elements of both orders $2$ and $3$.

Thus we may now assume that there are epimorphisms $\Delta\rightarrow C_i$ with the order of $C_i$ dividing $p_i^4-1$. From the previous paragraph both of $C_i\neq 1$. Since $\Gamma^{\ab} \cong \Z/3\Z$, we can assume $C_1$ say is cyclic of order $3$. Hence $\Delta\cap G_1^1$ is a normal subgroup of index $3$ which must coincide with $\rho([\G,\G])$. Since $[\G,\G]$ has abelianization $\Z/2\Z\times \Z/2\Z$ (see \cite[\S 2]{magma_calcs}), and using the nature of the filtration described above, the only possibility for the prime $p_1$ is $p_1=2$.  We now deal with $p_2$.  We have $\Delta\to C_2$ a cyclic group of order dividing $p_2^4-1$, with kernel $\Omega$ that surjects a cyclic group of $p_2$--power order.  Again, using $\Gamma^{\ab} \cong \Z/3\Z$, $C_2$ can only be a cyclic group of order $3$.  Since $\G$ admits a unique epimorphism to $\Z/3\Z$, it follows that is $\Omega=\rho([\G,\G])$ with $\Omega$ having a cyclic group of $p_2$--power order as a quotient.  Once again we 
 use the fact that $[\G,\G]$ has abelianization $\Z/2\Z\times \Z/2\Z$, and so the only possibility for $p_2$ is $p_2=2$. However, $2$ is inert to $\Q(\omega)$, that is to say there is only one prime of residue class degree a power of $2$, a contradiction.
\qed

\section{Profinite rigidity for $\G=\PSL(2,\Z[\omega])$}\label{s:rigid_bianchi}

In this section we exhibit the first example satisfying Theorem \ref{t:main}.

\begin{theorem}\label{main_bianchi}
$\G=\PSL(2,\Z[\omega])$ is profinitely rigid.
\end{theorem}

Before embarking on the proof, we will require some additional information about subgroups of low index and their abelianizations.

\subsection{Small index subgroups of $\G$ and their abelianizations}

We saw in Lemma \ref{l:abelG} that $\G$ has finite abelianization, in other words its first Betti number $b_1(\G)=0$. But since $\H^3/\G$ is a non-compact finite volume hyperbolic 3--orbifold, any finite sheeted {\em{manifold}} cover of $\H^3/\G$ has positive $b_1$; in other words, every torsion-free subgroup of $\G$ has positive $b_1$. We will exploit the following results about the abelianizations of subgroups of index at most $12$ in $\G$.

\begin{lemma}\label{indexatmost12}
If $\Delta < \G$ is a subgroup with $[\G:\Delta] \leq12$, then $b_1(\Delta)\leq 1$.
\end{lemma}

\begin{proof} 
We make some preliminary comments and then refer the reader to \cite[\S 2]{magma_calcs} for the Magma routine that completes the proof of the lemma. By Lemma \ref{torsion_normal}  $\G^{\ab}\cong \Z/3\Z$,  and so $\G$ has no subgroups of index $2$. Magma also shows that there are no subgroups of index $11$. In addition, Magma shows that there is a unique conjugacy class of subgroups of index $3$, $4$, $5$,  $9$ and $10$, two conjugacy classes of subgroups of index $6$ and $8$, and four of index $7$. There are $7$ of index $12$.  The Betti number for each of these subgroups is provided.
\end{proof}

The enumeration of small index subgroups shows that $\G$ has a unique conjugacy class of index $6$ subgroups with infinite abelianization, one class of index $10$, and three of index $12$. Only one of these conjugacy classes also contains $5$--torsion in its first homology group -- this class contains the index $12$ subgroup that we shall denote $\G_s$. This is the fundamental group of the once-punctured torus bundle $X$ that is known as the {\em{sister of the figure-eight knot complement}}. (The figure-eight knot complement itself represents the class of index $12$ subgroups with abelianization $\Z$, but we shall not pursue this.) For future reference we record that the monodromy of this once-punctured torus bundle for $\G_s$ is given by the matrix $\phi_s = \begin{pmatrix} -3& 1\cr -1& 0\cr\end{pmatrix}$, from which one easily computes $H_1(\G_s,\Z) = \Z \times \Z/5\Z$.

It is known that $\G_s$ is a congruence subgroup of $\G$ containing the principal congruence subgroup $\G(2)$ (see for example the proof of Lemma 3.1 of \cite{BaR}). From \S 3.2, we have that $\G/\G(2) \cong \PSL(2,\F_4)$, and so $[\G_s:\G(2)]=5$. It is well-known that $b_1(\G(2))=5$, and this can be confirmed by direct calculation of the abelianizations of the kernels of the maps from $\G_s$ onto the torsion part of $H_1(\G_s,\Z)$. Since none of the other subgroups $\Delta<\G$ of index at most $12$ with $b_1(\Delta)=1$ contains $5$--torsion in its abelianization, in these cases the only map from $\Delta$ to $\Z/5\Z$ is the one that factors through the unique epimorphism $\Delta\to\Z$. It is easy to compute the abelianization of the kernel of $\Delta\to\Z/5\Z$ using Magma (see \cite[\S 3]{magma_calcs}), and in each case one finds that the first Betti number is $1$. Summarizing this discussion, we have proved: 

\begin{proposition}\label{p:Sis-is-it}
Up to conjugacy, the only subgroup $\Delta<\G$ that fits into a chain $\Omega \triangleleft \Delta < \G$ with $[\G:\Delta]\le 12,\ [\Delta:\Omega]= 5,\ b_1(\Delta)\ge  1$ and $b_1(\Omega)\ge 5$ is $\Delta=\Gsis$.
\end{proposition}

We shall also need the following observation about the homology of cyclic covers of the once-punctured torus bundle with fundamental group $\Gsis$.

\begin{lemma}\label{l:ab_cyclic_covers}
If $d\ge 3$ then the abelianization of $F_2\rtimes_{\phi_s^d}\Z$ is $T\times\Z$, where $T$ is torsion and $|T|>5$.
\end{lemma}

\begin{proof} 
For any hyperbolic matrix $\psi$ of determinant $1$, a direct calculation of the abelianization of $G_\psi=F_2\rtimes_\psi\Z$ yields $\Z\times T_\psi$ where $\abs{T_\psi} = \abs{{\rm{tr}}(\psi)-2}$ (see \cite[Lemma 3.5]{BridR}, for example). Note that $\phi_s^2$ has trace $7$, and so we have $\abs{T}=5$, as it is for $\G_s$. 

Since $\phi_s^d$ is hyperbolic, the absolute value of its largest eigenvalue is strictly greater than $1$. It follows from elementary considerations that  $\abs{\mathrm{tr}(\phi_s^d)}$ is strictly monotonically increasing in $d$. Hence $\abs{\mathrm{tr}(\phi_s^d) - 2} \geq \abs{\mathrm{tr}(\phi_s^d)} - 2$  is greater than $5$ for all $d\geq 3$, as needed.
\end{proof}

\subsection{Proof of Theorem \ref{main_bianchi}}\label{s:bianchi_mainproof}

We turn to the proof of our main result. We are assuming that $\Delta$ is a finitely generated, residually finite group with $\wh{\Delta}\cong\wh{\Gamma}$. From \S \ref{s:rep} (see in particular Example \ref{ex1}), we obtain a homomorphism $\rho\colon\Delta\to\G$ whose image is Zariski-dense; we write $L$ to denote the image of $\rho$. Our purpose now is to show that $L=\G$ and that $\rho$ is injective.

The construction of $L$ from \S \ref{s:rep} provides considerable additional information about $L$. This guides our proof and can be used to shorten it (see Remark \ref{r:use_more}), but we attack the problem of showing $L=\G$ more directly and give an argument which shows that any non-elementary finitely generated subgroup of $\G$ has a finite quotient that $\G$ does not have (although the arguments are not phrased explicitly in this way). Our strategy is as follows. First we use arguments from 3--manifold topology to argue that if $L$ had infinite index in $\G$, then $L$  would have a subgroup of index at most $12$ with first Betti number greater than $1$. But $\G$  does not have such a subgroup (Lemma \ref{indexatmost12}), so $\wh{\G}$ cannot map onto $\wh{L}$. This brings us to the heart of the argument: the case where $L$ has finite index in $\Gamma$. In this case, we focus attention on the subgroup $\Lsis = L\cap\Gsis$ and use calculations of virtual Betti numbers to argue that if $\wh{\G}$ maps onto $\wh{L}$ then $\Lsis = \Gsis$. Considerations of co-volume and the Hopf property for finitely generated profinite groups then complete the proof.\\[\baselineskip]

\noindent{{\bf Notation:}} We focus on $\rho\colon\Delta\onto L$ and $\Lsis = L\cap\Gsis$. Define $\Deltasis = \rho^{-1}(\Gsis)$ and  $\Delta(2)=\rho^{-1}(\G(2))$.

\subsection{Ruling out infinite index image}\label{s:infindex}

\begin{lemma}\label{l:Lsis.b1}
$b_1(\Lsis)\le 1$.
\end{lemma}

\begin{proof} 
By definition, $\Lsis$ is a quotient of $\Deltasis=\rho^{-1}(\Gsis)$, so $b_1(\Lsis)\le b_1(\Deltasis)$. But $\Deltasis$ has index at most $12$ in $\Delta$, so by the Correspondence Theorem (Proposition \ref{correspondence}), it has the same abelianization as some subgroup of index at most 12 in $\G$. Lemma \ref{indexatmost12} tells us that these subgroups of $\G$ all have first Betti number at most one. Thus $b_1(\Deltasis)\le 1$.
\end{proof}

\begin{proposition}\label{p:not-inf}
$L$ has finite index in $\G$.
\end{proposition}

\begin{proof}
By construction, $L$ is Zariski-dense. In particular neither $L$ nor $\Lsis=L\cap\Gsis$, which has index at most $12$ in $L$, is abelian. Since $L$ is finitely generated, so is $\Lsis$. And since $\Gsis$ is torsion-free, so is $\Lsis$. Thus $M=\H^3/\Lsis$ is a non-elementary orientable hyperbolic manifold with finitely generated fundamental group. Classical $3$--manifold topology provides a compact core for $M$ (\cite{Scott}), i.e.~a compact 3--manifold with boundary $N\subset M$ that is homotopy equivalent to $M$. If $\Lsis$ were of infinite index in $\Gsis$, then $M$ would have infinite volume and therefore a big end, i.e.~at least one of the connected components of $\partial N$ would have genus at least $2$. A standard duality argument establishes the following well-known ``half lives, half dies" principle:

\medskip
\noindent{\em If $N$ is a compact orientable 3--manifold with non-empty boundary $\partial N$, then the kernel and image of the natural map $H_1(\partial N, \mathbb R)\to H_1(N,\mathbb R)$ are of equal dimension. In particular, if $\partial N$ has a component of genus at least $2$, then $b_1(N)\ge 2$.}
\medskip

Thus if $L$ were of infinite index in $\G$ then we would have $b_1(\Lsis) = b_1(N)\ge 2$, contradicting Lemma \ref{l:Lsis.b1}.
\end{proof}

\subsection{$\Deltasis=\Gsis$}

At this stage we have a finitely generated, residually finite group $\Delta$ with $\wh{\Delta}\cong\wh{\G}$ mapping onto a subgroup of finite index $L<\G$. We focus on $\Deltasis$, the preimage of $L\cap\Gsis$.
 
\begin{lemma}\label{l:sissy} 
$\wh{\Deltasis} \cong \wh{\Gsis}$. Moreover, the subgroups of index $12$ in $\Delta$ corresponding to the conjugates of $\Gsis$ are the conjugates of $\Deltasis$.
\end{lemma}

\begin{proof} 
As $\Lsis$ has finite index in $\Gsis$, we have $b_1(\Lsis)\ge b_1(\Gsis) =1$. From Lemma \ref{l:Lsis.b1} we deduce $b_1(\Lsis)=1$. And as $\Lsis\cap \G(2)$ has finite index in $\G(2)$, we have $b_1(\Lsis\cap\G(2))\ge  b_1(\G(2))=5$. Now, $\Lsis\cap \G(2)$ is normal in $\Lsis$ with index $1$ or $5$, and in fact the index must be $5$ since the Betti number of the subgroup is greater than that of $\Lsis$. Thus $\Deltasis = \rho^{-1}(\Gsis)$ is a subgroup of index at most $12$ in $\Delta$ with $b_1(\Deltasis)$ positive; moreover it contains $\Delta(2)=\rho^{-1}(\G(2))$ as a subgroup of index $5$ and $b_1(\Delta(2))\ge 5$. So in $\Delta$ we have a chain of subgroups $\Delta(2) \ns \Deltasis < \Delta$ with $[\Delta:\Deltasis]\le 12,\ [\Deltasis:\Delta(2)]= 5,\ b_1(\Deltasis)\ge  1$ and $b_1(\Delta(2))\ge 5$.

We pass this chain across to a chain of subgroups in $\G$ using Proposition \ref{correspondence}. Explicitly, replacing each of the groups $\Delta$, $\Delta_s$, and $\Delta(2)$ in this sequence with $\overline{\Delta}\cap\G$, $\overline{\Delta_s} \cap \G$, and $\overline{\Delta(2)} \cap \G$, we obtain a chain of subgroups in $\G$ with the same properties. From Proposition \ref{p:Sis-is-it} we deduce that, up to conjugacy in $\G$, we have $\overline\Deltasis\cap \G = \Gsis$. In particular, $\wh{\Deltasis}\cong \wh{\Gsis}$.
\end{proof}

\subsection{The Final Argument}

We now have $\wh{\Gsis}\cong \wh{\Deltasis}$ with $\Deltasis$ mapping onto the finite index subgroup $\Lsis < \Gsis$. The description of $\Gsis$ as a punctured-torus bundle gives a short exact sequence $1\to F \to \Gsis \to \Z\to 1$, where $F$, a free group of rank $2$, is the fundamental group of the fibre. This restricts to a short exact sequence $1\to F_L \to \Lsis \to \Z\to 1$ where $F_L=F\cap L$. At this point it is important to note that $\Gsis$ induces the full profinite topology on $F$ and that $\Lsis$ induces the full profinite topology on $F_L$ (see for example \cite[Lemma 2.2]{BridR}).
 
Since $b_1(\Lsis)$ is one, $F_L$ is the kernel of the unique map $\Lsis\to \Z$. Similarly, $\wh{F}<\wh{\Gsis}$ is the kernel of the unique epimorphism $\wh{\Gsis}\to \wh{\Z}$ and $\wh{F_L}<\wh{\Lsis}$ is the kernel of the unique epimorphism  $\wh{\Lsis}\to \wh{\Z}$, so the  epimorphism $\wh{\Gsis}\to \wh{\Lsis}$ must send $\wh{F}$ onto $\wh{F_L}$. Therefore the free group $F_L$ has rank $2$, and since it has finite index in $F$, we conclude that $F_L=F$; in other words $F<L$. As $\Lsis$ contains the fibre group $F<\Gsis$, it is the fundamental group of a cyclic covering of $\H^3/\Gsis$. Algebraically, $\Gsis = F\rtimes_{\phi_s}\Z$ and $\Lsis = F\rtimes_{\phi_s^d}\Z$. If $d>2$, then it follows from Lemma \ref{l:ab_cyclic_covers}, that $\Lsis$ would have a finite abelian quotient that $\Gsis$ does not have, hence in this case we can conclude that $\Lsis=\Gsis$. 

We now deal with the case $d=2$. Assume that $\Lsis = F\rtimes_{\phi_s^2}\Z$.  In this case $\Lsis$ has a unique $2$--fold cover $N$, for which the order of $T$ (the torsion subgroup of $H_1(N,\Z)$) is $45$. But the index $2$ subgroup of $\Gsis$ must surject $\pi_1(N)$ and this cannot happen as this index $2$ subgroup only has torsion subgroup of order $5$. At this stage we know that $\Gsis = \Lsis \le L \le \G$. Moreover, $[L:\Lsis]=12$ because $\Deltasis = \rho^{-1}(\Lsis)$ has index $12$ in $\Delta$ (as it corresponds to $\Gsis$). As $[\G:\Gsis]=12$, we conclude that $L=\G$.

Finally, we have $\wh{\G}\cong \wh{\Delta}\overset{\hat{\rho}}\to\wh{L}=\wh{\G}$, and as finitely generated profinite groups are Hopfian (see \cite[Prop 2.5.2]{RZ}), we conclude that $\wh{\rho}$ and $\rho$ are injective. Thus $\rho\colon \Delta\to L = \G$ is an isomorphism.
$\square$

\begin{remark} \label{r:use_more} 
We close this section by explaining our earlier comment that \S \ref{s:rep}  can be used to avoid parts of the above analysis of subgroups of $\G$. The point is to exploit more explicitly the fact that $\G_s$ contains the principal congruence subgroup $\G(2)$. From \S  \ref{Paoluzzi_Zimmermann} we have that $\G/\G(2) \cong \PSL(2,\F_4)$, and so $[\G_s:\G(2)]=5$. Moreover, by Theorem \ref{controllingmodp}, there is a unique $ \PSL(2,\F_4)$ quotient of $\G$. From the discussion in \S \ref{congruence_quotients} we have a subgroup $L(2)<L$ with $L/L(2)\cong \G/\G(2) \cong \PSL(2,\F_4)$, and $L(2)=L\cap \Gamma(2)$. In addition,  \S \ref{congruence_quotients}  also provides us with a subgroup $L_s=L\cap \G_s$ which is a subgroup of index $12$ in $L$ containing $L(2)$ of index $5$.
 
Using the epimorphism $\Delta \rightarrow L$ and the isomorphism $\wh\Delta\cong\wh\G$, we have an epimorphism $\phi:\wh{\G}\rightarrow \wh{L}$, and hence there is a subgroup $\Pi_\G<\G$ of index $12$ with $\phi(\wh{\Pi_\G})=\wh{L}_s$.  As above, it can be shown that the only possibility for $\Pi_\G$ is $\Pi_\G=\G_s$.
\end{remark}

\begin{remark}[Exhibiting Additional Quotients] 
The techniques of this section enable one to show that if $\Delta<\G=\PSL(2,\Z[\omega])$ is a proper finitely generated infinite subgroup, then $\Delta$ has a finite quotient that $\G$ does not have: there does not exist a continuous epimorphism $\widehat{\G}\twoheadrightarrow\wh{\Delta}$. When $\Delta$ has infinite index, alternative techniques allow one to say something more: $\Delta$ will have a congruence quotient ${\rm{PSL}}(2,\mathbb{F})$ that $\G$ does not have. This will follow from a more general result that we prove in \cite{BMRS2} using techniques from Teichm\"uller theory and the study of character varieties. This more general result is a component in our proof of absolute profinite rigidity for certain cocompact Fuchsian triangle groups. \end{remark}

\section{Profinite rigidity for the lattices containing $\PSL(2,\Z[\omega])$}\label{s:others}

In this section we explain how Mostow Rigidity can be used to promote the profinite rigidity of $\PSL(2,\Z[\omega])$ to profinite rigidity for each of the lattices in ${\rm{Isom}}(\H^3)$ that contain it; these lattices were described in \S \ref{s:small}. To maintain brevity, throughout this section $\G$ will always denote $\PSL(2,\Z[\omega])$.
  
\begin{lemma}\label{l:up2}
If $\Delta$ is a group with a subgroup of index 2 that is isomorphic to $\G$, then $\Delta$ is isomorphic to one of  $\G_0, \Lambda_1$, $\Lambda_2$ or $\G\times\Z/2\Z$.
\end{lemma}

\begin{proof} 
By Mostow Rigidity, $\Lambda:={\rm{Aut}}(\G)$ is a lattice in $\PSL(2,\C)$ containing $\G$ (which has trivial centre) as the group of inner automorphisms -- a subgroup of finite index. The action of $\Delta$ by conjugation on $\G$ defines  a homomorphism $\Delta \to \Lambda $ whose image contains $\G$ and whose kernel $\Omega$ commutes with $\G$. If $\Omega=1$ then $\Delta< \Lambda$ is a lattice and Lemma \ref{l:upLattice} completes the proof. Otherwise $\Omega$ has order $2$, the image of $\Delta$ in $\Lambda$ is $\G$, and the sequence splits.
\end{proof}

\begin{theorem}\label{t:rigid2} 
Each of the groups $\G_0$, $\Lambda_1$, $\Lambda_2$ and $\G\times\Z/2\Z$ is profinitely rigid.
\end{theorem}

\begin{proof} 
Let $\Omega_0$ be a finitely generated, residually finite group that has the same profinite completion as one of the groups $\G_0, \ \Lambda_1$, $\Lambda_2$ or $\G\times\Z/2\Z$. Then, by the Proposition \ref{correspondence},  $\Omega_0$ has a subgroup $\Omega$ of index $2$ with the same profinite completion as $\G$. Since $\G$ is profinitely rigid, $\Omega \cong \G$, so Lemma \ref{l:up2} tells us that $\Omega_0$ is isomorphic to one of $\G_0, \ \Lambda_1$, $\Lambda_2$ or $\G\times\Z/2$. From Lemma \ref{l:abelG}, the first two of these groups have abelianization $\Z/2\Z$, and $\G\times\Z/2\Z$ has abelianization $\Z/6\Z$. From \S \ref{s:small}, 
$\Lambda_2$ also abelianization $\Z/6\Z$. Hence we can distinguish $\G_0$ and $\Lambda_1$ from $\Lambda_2$ and $\G\times \Z/2\Z$ by abelianization.  It remains to distinguish $\G_0$ from $\Lambda_1$, and $\Lambda_2$ and $\G\times \Z/2\Z$. We can distinguish $\wh\G_0$ from $\wh{\Lambda}_1$ using Magma (see \cite[\S 4]{magma_calcs}) to count the number of conjugacy classes of index $8$ subgroups: in $\Lambda_1$ there is only one conjugacy class, whereas in $\G_0=\PGL(2,\Z[\omega])$ there are three.  It follows from Lemma \ref{l:conj-classes} that  $\wh\G_0\not\cong\wh{\Lambda}_1$. 

Finally, referring to \cite[\S4] {magma_calcs}, we see that $\Lambda_2$ has no subgroups of index $7$, whereas $\G\times \Z/2\Z$ has $4$ conjugacy classes of such subgroups.
\end{proof} 

\begin{theorem}\label{t:rigid4} 
$\Lambda_0$ is profinitely rigid.
\end{theorem}

\begin{proof} 
As $\G_0$, which has index 2 in $\Lambda_0$, is profinitely rigid, by repeating the argument of Lemma \ref{l:up2} we deduce that a finitely generated, residually finite group with the same profinite completion as $\Lambda_0$ is either a non-uniform lattice of covolume $v_0/24$ or is isomorphic to $\G_0\times\Z/2\Z$. As $\Lambda_0$ is the only non-uniform lattice of covolume $v_0/24$ up to conjugacy in ${\rm{Isom}}(\H^3)$, we need only distinguish the profinite completions of $\Lambda_0$ and $\G_0\times\Z/2\Z$. As above we can use Magma (see \cite[\S 5]{magma_calcs}) to show that $\Lambda_0$ and $\G_0\times\Z/2\Z$ have different numbers of conjugacy classes of subgroups of index $8$, and so it follows from Lemma \ref{l:conj-classes} that  $\wh{\Lambda}_0\not\cong\wh{\G}_0 \times \Z/2\Z$. 
\end{proof}

\section{Profinite rigidity of the Weeks manifold}\label{s:weeks_rigidity}

In this section we provide our second main example of a profinitely rigid arithmetic Kleinian group, namely $\G_W$ (as described in \S \ref{weeks})

\begin{theorem}\label{main_weeks}
$\Gamma_W$ is profinitely rigid.
\end{theorem}

Throughout this section $\Delta$ is a finitely generated residually finite group with $\widehat{\Delta} \cong \widehat{\G}_W$. We can quickly reduce consideration to the following situation.

\begin{lemma}\label{l: reduce_weeks} 
With $\Delta$ as above, there exists a finite index subgroup $L<\G_W$ with $\rho\colon\Delta \onto L$. 
 \end{lemma}

\begin{proof} 
From Example \ref{ex2}, together with the arithmetic description of  $\G_W$ in \S \ref{weeks} and Proposition \ref{reidwang}, we deduce that there exists a maximal order $\mathcal{O} \subset B_W$ with $\rho\colon\Delta \onto L$ a finitely generated subgroup of $\G_\mathcal{O}^1$.  Also from \S \ref{weeks},  $\G_\mathcal{O}^1$ contains $\G_W$ as a normal subgroup of index $3$. We first claim that $L<\Gamma_W$. Indeed, if $L$ were not contained in $\G_W$, then $L$ would map onto $\Gamma_\mathcal{O}^1/\G_W$, which, from the above remark is isomorphic to $\Z/3\Z$. However, this is impossible as $L$ is a quotient of $\Delta$ and $H_1(\Delta,\Z)\cong H_1(\G_W,\Z) \cong {\Bbb Z}/5{\Bbb Z}\times {\Bbb Z}/5{\Bbb Z}$ (recall \S \ref{weeks}). Since $L < \G_W$, the quotient $\mathbb{H}^3/L$ is a manifold. If $L$ were of infinite index then an application of ``half lives, half dies"  (recall the proof of Lemma \ref{p:not-inf}) would imply that $L$ has infinite abelianization, which it can not since it is a quotient of $\Delta$.
\end{proof}

The proof of Theorem \ref{main_weeks} follows a similar strategy to that of the ``finite index case" of Theorem \ref{main_bianchi}; namely, we identify a particular fibered cover and combine the study of it with an analysis of low index subgroups. The next two subsections detail what we need in this direction.

\subsection{Subgroups of index $24$} \label{index23}

In this section $K$ denotes the trace-field of $\G_W$ and $R_K$ its ring of integers (recall \S \ref{weeks}). It can be checked that $R_K$ contains two prime ideals of norm $23$, one corresponding to the ramified prime of norm $23$, which we denote by $\mathcal{Q}$, and a second unramified prime which we denote by $\mathcal{P}$.  For both  we have $\PSL(2,R_K/\mathcal{Q})\cong \PSL(2,R_K/\mathcal{P})\cong \PSL(2,{\Bbb F}_{23})$. We claim that, up to conjugacy, $\Gamma_W$ has two epimorphisms to $\PSL(2,{\Bbb F}_{23})$, and these arise as $\PSL(2,R_K/\mathcal{Q})$ and $\PSL(2,R_K/\mathcal{P})$. The existence of the epimorphisms is immediate from the inclusion of $\Gamma_W$ into $\Gamma_\mathcal{O}^1$ as a normal subgroup of index $3$, since $\PSL(2,{\Bbb F}_{23})$ is simple. For uniqueness, a theorem going back to Galois states that the minimal index of any proper subgroup of the simple group $\PSL(2,{\Bbb F}_{23})$ is $24$. Moreover, there is a unique conjugacy class of subgroups of index $24$. Thus every epimorphism $\G_W\onto \PSL(2,{\Bbb F}_{23})$ gives rise to a subgroup of index $24$ with normal core $\Omega$ such that $\Gamma_W/\Omega \cong \PSL(2,{\Bbb F}_{23})$. In \cite[\S 6]{magma_calcs} we provide the Magma calculations enumerating all subgroups of index $24$ in $\G_W$ (up to conjugacy in $\G_W$). From this we see that there $11$ such subgroups and two have normal core $\Omega$ such that $\abs{\Gamma_W/\Omega} =6072 = \abs{\PSL(2,{\Bbb F}_{23})}$.  Moreover, the Magma routine also shows that these finite quotients are simple and hence must be isomorphic to $\PSL(2,{\Bbb F}_{23})$ (using the lists of simple groups of small order).

\subsection{A fibered cover of $M_W$}\label{fiberedcover}

We shall make use of an explicit cover of $M_W$ that is a genus $2$ surface bundle. The existence of a bundle cover of $M_W$ was exhibited by Button \cite{Butt}, but we require more detailed information about this cover.  From the tables of \cite{Butt} we see that $M_W$ is commensurable with the fibered manifold $M=m289(7,1)$, which arises from surgery on the census manifold $m289$ from the SnapPy census \cite{CDW}.  In fact, using the identification of certain of these census manifolds with knots in the tables through $9$ crossings \cite{CalDW}, one knows that the manifold $m289$ is homeomorphic to the complement of the knot $\mathcal{K}=6_2$ (shown below). Thus $M_W$ is commensurable with $0$--surgery on $S^3\setminus \mathcal{K}$ (the framing used by SnapPy is different from the standard one for the knot $\mathcal{K}$).  

\begin{figure}[h]
\centering

\begin{tikzpicture}[scale=1]
  \draw [thick] (-1.25,0) to [out=90,in=135] (-.25,.25);
  \draw [thick] (.25,-.25) to [out=-45,in=-135] (1.5,-1);
  \draw [thick] (.8,.6) to [out=135,in=35] (-.4,1.2);
  \draw [thick] (-.7,0) to [out=-90,in=145] (-.4,-1.2);
  \draw [thick] (.8,-.6) to [out=45,in=-45] (1.5,1);
  \draw [thick] (.25,.25) to [out=-135,in=45] (-.25,-.25);
  
  \draw [white , line width = 5pt] (-.25,.25) to [out=-45,in=135] (.25,-.25);
  \draw [thick] (-.25,.25) to [out=-45,in=135] (.25,-.25);
  
  \draw [white , line width = 5pt] (1.5,-1) to [out=45,in=-45] (.8,.6);
  \draw [thick] (1.5,-1) to [out=45,in=-45] (.8,.6);
  
  \draw [white , line width = 5pt] (-.4,1.2) to [out=-145,in=90] (-.7,0);
  \draw [thick] (-.4,1.2) to [out=-145,in=90] (-.7,0);
  
  \draw [white , line width = 5pt] (-.4,-1.2) to [out=-35,in=-135] (.8,-.6);
  \draw [thick] (-.4,-1.2) to [out=-35,in=-135] (.8,-.6);
  
  \draw [white , line width = 5pt] (1.5,1) to [out=135,in=45] (.25,.25);
  \draw [thick] (1.5,1) to [out=135,in=45] (.25,.25);
  
  \draw [white , line width = 5pt] (-.25,-.25) to [out=-135,in=-90] (-1.25,0);
  \draw [thick] (-.25,-.25) to [out=-135,in=-90] (-1.25,0);
  
\end{tikzpicture}

\caption{}
\end{figure}

Since $\mathcal{K}$ is $2$--bridge (and hence alternating) and its Alexander polynomial is $t^4-3t^3+3t^2-3t+1$,  $S^3\setminus \mathcal{K}$ fibers over the circle with fibre a once-punctured surface of genus $2$ (\cite{Kan}). Using \cite{Knot}, the monodromy, $\psi$ of this fibration can be described by the composition of Dehn twists $T_a\circ T_b\circ T_c\circ T_d^{-1}$ (see Figure 3 for the labelling of curves). Here, $T_\gamma$ denotes the right-handed twist in $\gamma$.

\begin{figure}[h]
\centering

\begin{tikzpicture} [scale=1.2]

  \draw [semithick] (-3,0) to [out=90,in=170] (-1.5,1) to [out=-10,in=180] (0,.5) 
  	to [out=0,in=190] (1.5,1) to [out=10,in=90] (3,0) to [out=-90,in=-10] (1.5,-1) 
  	to [out=170,in=0] (0,-.5) to [out=180,in=10] (-1.5,-1) to [out=190,in=-90] (-3,0);
  
  \filldraw (2.8,0) circle (1.5pt);
 
  \draw [semithick] (1.6,-.25) arc [start angle=-90, end angle = 0,
    x radius = 6.5mm, y radius = 4mm];
  \draw [semithick] (1.6,-.25) arc [start angle=-90, end angle = -180,
    x radius = 6.5mm, y radius = 4mm];
  \draw [semithick] (1.6,.25) arc [start angle=90, end angle = 22,
    x radius = 6.5mm, y radius = 4mm];
  \draw [semithick] (1.6,.25) arc [start angle=90, end angle = 158,
    x radius = 6.5mm, y radius = 4mm];
  
  \draw [semithick] (-1.6,-.25) arc [start angle=-90, end angle = 0,
    x radius = 6.5mm, y radius = 4mm];
  \draw [semithick] (-1.6,-.25) arc [start angle=-90, end angle = -180,
    x radius = 6.5mm, y radius = 4mm];
  \draw [semithick] (-1.6,.25) arc [start angle=90, end angle = 22,
    x radius = 6.5mm, y radius = 4mm];
  \draw [semithick] (-1.6,.25) arc [start angle=90, end angle = 158,
    x radius = 6.5mm, y radius = 4mm];
    
  \draw [line width = .8pt] (-1.6,0) circle [x radius = 10mm, y radius = 5mm];
  \draw node at (-2,-.65) {$a$};
  
  \draw [line width = .8pt] (1.6,0) circle [x radius = 10mm, y radius = 5mm];
  \draw node at (2.6,-.4) {$c$};
  
  \draw [line width = .8pt] (0,.2) arc [start angle=90, end angle = 0,
    x radius = 10mm, y radius = 2mm];
  \draw [line width = .8pt] (0,.2) arc [start angle=90, end angle = 180,
    x radius = 10mm, y radius = 2mm];
  \draw [line width = .8pt, densely dashed] (0,-.2) arc [start angle=-90, end angle = 0,
    x radius = 10mm, y radius = 2mm];
  \draw [line width = .8pt, densely dashed] (0,-.2) arc [start angle=-90, end angle = -180,
    x radius = 10mm, y radius = 2mm];
  \draw node at (0,.3) {$b$};
  
  \draw [line width = .8pt] (1.6,.25) arc [start angle=-90, end angle = 90,
    x radius = 2mm, y radius = 3.8mm];
  \draw [line width = .8pt, densely dashed] (1.6,.25) arc [start angle=-90, end angle = -270,
    x radius = 2mm, y radius = 3.8mm];
  \draw node at (1.6,1.2) {$d$};
  
  \draw [line width = .8pt] (-1.6,.25) arc [start angle=-90, end angle = 90,
    x radius = 2mm, y radius = 3.8mm];
  \draw [line width = .8pt, densely dashed] (-1.6,.25) arc [start angle=-90, end angle = -270,
    x radius = 2mm, y radius = 3.8mm];
  \draw node at (-1.6,1.2) {$f$};
  
  \draw [line width = .8pt] (1.6,-.25) arc [start angle=90, end angle = -90,
    x radius = 2mm, y radius = 3.8mm];
  \draw [line width = .8pt, densely dashed] (1.6,-.25) arc [start angle=90, end angle = 270,
    x radius = 2mm, y radius = 3.8mm];
  \draw node at (1.6,-1.2) {$e$};
  
\end{tikzpicture}

\caption{}
\end{figure}

One computes the action of $\psi$ on the homology of the fiber with respect to the basis $\{a,f,c,d\}$ is given by the matrix $A$ shown below.

\[ A = \begin{pmatrix} 0 &  1 & 1 & 1\cr -1 & 1 & 2 & 1\cr 0 & 0  & 2 & 1\cr 1 & 0 & -1 & 0\cr\end{pmatrix}~\hbox{and}~ A^6 =   \begin{pmatrix} 18 &  17 & 88 & 57\cr 9 & 9 & 48 & 31\cr 14 & 12  & 66 & 43 \cr 3 & 2 & 9 & 6\cr \end{pmatrix}. \]                                                                                                                                    
The fibration of $S^3\setminus \mathcal{K}$ extends to the surgered manifold $M$ and the action of the monodromy on the homology of the closed genus $2$ surface is again given by $A$.  Using SnapPy \cite{CDW}, $M$ can be shown to be hyperbolic of volume approximately $3.77082945\ldots$. 

We will be interested in the $6$--fold cyclic covering of $M$, which we denote by $M_6$. Our interest lies with the fact that $M_6$ arises as an index $24$ cover of the Weeks manifold. One could verify this by deriving presentations of the index $24$ subgroups of $\G_W$, on the one hand, while on the other hand calculating a presentation of $\G_6:=\pi_1(M_6)$ from its description as the $6$--fold cyclic cover of the surgered manifold $m289$; a package such as Magma could then verify that the groups are isomorphic, and Mostow Rigidity then assures us that the manifolds are the same. But such calculations would leave the reader in the dark as to why these facts are true, so instead we shall explain, with references, the structure that leads to this conclusion. The following lemma gathers the key facts from the preceding discussion and its proof contains the promised explanation. 

\begin{lemma}\label{bundlecover}
Let $M$ denote the manifold obtained by $0$--surgery on $S^3\setminus \mathcal{K}$. Then
\begin{enumerate}
\item $S^3\setminus \mathcal{K}$ is fibered with fiber a once-punctured genus $2$ surface, 
and $M$ is fibered with fiber a genus $2$ surface.
\item The action of the monodromy on the homology of the fiber for $S^3 \setminus \mathcal{K}$ is given by $A$ above.
\item  $M_6$, the $6$--fold cyclic cover of $M$, is fibered with fiber a genus $2$ surface and $H_1(M_6,{\Bbb Z})\cong {\Bbb Z}\times {\Bbb Z}/5{\Bbb Z}\times {\Bbb Z}/55{\Bbb Z}$. 
\item $M_6$ is a $24$--fold cover of $M_W$. 
\item $\G_6=\pi_1(M_6)$ lies in the conjugacy class of subgroup $l[1]$ of \cite[\S 6]{magma_calcs}.
\end{enumerate}
\end{lemma}

\begin{proof}  
The proof of the first two items is contained in the preceding discussion. For the third part, $M_6$ is clearly fibered with fiber $\Sigma$ a once-punctured surface of genus $2$, and one calculates $H_1(M_6,\Z)$ by taking the quotient of action of $A^6$ on $H_1(\Sigma,\Z) = \Z^4 = \<a,f,c,d\>$. (The torsion subgroup of $H_1(M_6,\Z)$ has order $275$.) For the next part let $M={\Bbb H}^3/\Omega$, and note that by \cite{Butt} the lattice $\Omega$ is arithmetic and commensurable with $\G_W$. However, $\Omega$ is not derived from a quaternion algebra. Indeed it can be checked using Snap \cite{snap} that its trace-field has degree $6$ and $\Omega^{(2)} <\G_\mathcal{O}^1$. If $\Omega^{(2)}<\Gamma_W$ then, by volume considerations, the index would be $8$. However, a Magma calculation (see \cite[\S 6]{magma_calcs}) shows that $\G_W$ has a unique conjugacy class of subgroups of index $8$ and the corresponding manifold is a rational homology $3$--sphere with first homology group ${\Bbb Z}/5{\Bbb Z}\times {\Bbb Z}/30{\Bbb Z}$; in particular it cannot be a fiber bundle over the circle. Thus $\Omega^{(2)}\cap \Gamma_W$ must have index $3$ in $\Pi^{(2)}$. The double cover of $S^3\setminus \mathcal{K}$ has first homology group ${\Bbb Z}\times {\Bbb Z}/11{\Bbb Z}$. Performing $0$--surgery on this manifold produces ${\Bbb H}^3/\Omega^{(2)}$.  In particular, the first homology group is ${\Bbb Z}\times {\Bbb Z}/11{\Bbb Z}$, and so there is a unique homomorphism onto ${\Bbb Z}/3{\Bbb Z}$ whose kernel provides the cover $M_6$. Volume considerations show that the covering degree $M_6\to M_W$ is $24$. For the last part, by inspection of the lists of first homology groups in \cite[\S 6]{magma_calcs}, the only possibility for $\Gamma_6$ is the group $l[1]$.
\end{proof}

We will also need the following.

\begin{lemma}\label{inf_order}
Let $\Mod_g$ be the Mapping Class group of the closed orientable surface of genus $g$. Let $\eta\in \Mod_g$ and let $G(\eta) = S_g\rtimes_\eta\Z$ be the fundamental group of the bundle with holonomy $\eta$ and assume $b_1(G(\eta))=1$. If the image of $\eta\in \Mod_g$ under the natural homomorphism $\Mod_g\to{\rm{Sp}}(2g,\mathbb{Z})$ has infinite order and $d>1$, then $\widehat{G(\eta)}\not\cong\widehat{G(\eta^d)}$.
\end{lemma}

\begin{proof} 
The congruence topology on ${\rm{Sp}}(2g,\mathbb{Z})$ induces the full profinite topology on abelian subgroups; see Segal \cite[Ch 10]{segal} (see also \cite{McR}). Thus, given $d$, there is an integer $n_0$ such that the images of $\eta$ and $\eta^d$ generate distinct cyclic subgroups of ${\rm{Sp}}(2g,\mathbb{Z}/n_0\mathbb{Z})$. As $b_1(G(\eta))=1$, we can now argue as in Lemma 2.5 of \cite{BRW}; we recall the details. The unique epimorphism $G(\eta)\to\Z$ defines a short exact sequence
\[ 1\to \wh{S_g}\to \wh{G(\eta)}\to \wh{\Z}\to 1. \]
If $\Omega<S_g$ is a characteristic and of finite index, then the canonical map $S_g\to S_g/\Omega$ defines an epimorphism $\wh{G(\eta)}\to G(\eta)/\Omega$. As $\wh{\Omega}$ is normal in $\wh{G(\eta)}$, the action of $\wh{\Z}$ by conjugation on $\wh{G(\eta)}$ descends to an outer action on $\wh{S_g}/\wh{\Omega}=S_g/\Omega$, defining a cyclic subgroup  $C_\eta<\Out(G/\Omega)$. The righthand factor of $G(\eta)=S_g\rtimes_\eta\Z$ is dense in $\wh{\Z}$, so the image of $\eta$ generates $C_\eta$. If $\Omega$ is the kernel of the canonical map $S_g\to H_1(S_g,\Z/n_0\Z)$, then $C_\eta=C_\eta(n_0)$ is the cyclic group generated by the image of $\eta$ in ${\rm{Sp}}(2g,\Z/n_0\Z)<{\rm{Out}}(\Z/n_0\Z)^{2g}$. By construction,  $C_\eta(n_0)$ is an invariant of $\wh{G(\eta)}$ (rather than $G(\eta)$) and $\abs{C_{\eta^d}(n_0)} < \abs{C_\eta(n_0)}$.  Thus $\widehat{G(\eta)}\not\cong\widehat{G(\eta^d)}$.
\end{proof}

\subsection{Proof of Theorem \ref{main_weeks}}

Let $\Gamma_6=\pi_1(M_6)$ be as in Lemma  \ref{bundlecover}. Then $[\G_W:\G_6]=24$ and  $\G_6$ is the subgroup $l[1]$ described in Lemma \ref{bundlecover}(5). Denoting  the core of $\G_6$ in $\G_W$ by $\Omega$, we have $\G_6/\Omega\cong \PSL(2,{\F}_{23})$.
Let $L_6=\Gamma_6\cap L$ and $L_\Omega=\Omega\cap L$. From the discussion in \S \ref{congruence_quotients} we know that $L/L_\Omega\cong\Gamma_W/\Omega\cong \Delta/\rho^{-1}(\Omega)$ and that the epimorphism $L\to L/L_\Omega \cong \PSL(2,{\Bbb F}_{23})$ is the restriction of the epimorphism $\Gamma_W\to \Gamma_W/\Omega$.  Note that $L\neq L_6$ since $b_1(L_6)>0$, and $L$ can only have finite abelian quotients (since it is a quotient of $\Delta$ and $\wh{\Delta} \cong\wh{\G}_W$). It follows that $[L:L_6]=24$, since on the one hand $[L:L_6]\leq [\Gamma_W:\Gamma_6]=24$, while on the other hand, as remarked above, $24$ is the minimal index of a proper subgroup of $\PSL(2,{\Bbb F}_{23})$.  

Now consider $\Delta_6=\rho^{-1}(L_6)<\Delta$. This has index $24$ in $\Delta$ and moreover by construction of $L$ (see \S \ref{congruence_quotients}) the normal core of $\Delta_6$ in $\Delta$ is $\rho^{-1}(\Omega)$. From \cite[\S 6]{magma_calcs},  and Proposition \ref{correspondence}, there is a unique subgroup (up to conjugacy in $\Delta$) of index $24$ with normal core having quotient $\PSL(2,{\Bbb F}_{23})$ and positive first Betti number. By construction, $\Delta_6$ is this unique subgroup and hence $\wh{\G}_6 \cong \wh{\Delta}_6$. 

Consider the epimorphism $\wh{\G}_6 \cong \wh{\Delta}_6\to \wh{L}_6$. Note that $b_1(L_6)=1$ as $b_1(\G_6)=1$. By construction, $N={\H}^3/L_6$ is a closed hyperbolic $3$--manifold that fibers over the circle: it is a finite cover of the fibered manifold $M_6$. By Lemma \ref{bundlecover}, $M_6$ is a genus $2$ surface bundle over the circle; let $F<\G_6$ be the fundamental group of the fiber. The fiber of $N$ has fundamental group $L_6\cap F$. When we take profinite completions, the epimorphism $\wh{\Gamma}_6 \cong \wh{\Delta}_6\to \wh{L}_6$ restricts to an epimorphism $\wh{F}\to \wh{L_6\cap F}$ (because $F$ and $L_6\cap F$ are the respective kernels of the unique epimorphisms to $\Z$). It follows that $L_6\cap F$ has genus at most $2$, and hence $L_6\cap F= F$. In particular $N\to M_6$ is a cyclic cover (of degree $d$ say) and $L_6 = F\rtimes_{\psi^d}\Z$.

Now, since $\wh{F}$ is Hopfian, we have that $\wh{\rho}$ restricted to $\wh{F}<\wh{\G}_6$ is injective. Elements of $\wh{\G}_6$ in the complement of $\wh{F}$ project non-trivially to the right-hand factor of $\wh{\G}_6=\wh{F}\rtimes\wh{\Z}$ and hence map non-trivially under $\wh{\rho}$ to the  $\wh{\Z}$ factor of $\wh{L_6}= \wh{F}\rtimes\wh{\Z}$. We conclude that $\wh{\rho}$ is injective on $\wh{\Delta}_6$ and hence we have an isomorphism $\wh{\G_6}\cong \wh{\Delta}_6\to \wh{L}_6$. Thus  $\wh{F\rtimes_\psi\Z}\cong\wh{F\rtimes_{\psi^d}\Z}$. Since  the action of $\psi$ on $H_1(F,\Z)$ is represented by a positive matrix, it has infinite order in ${\rm{Sp}}(4,\Z)$, so Lemma \ref{inf_order} applies and we conclude that $d=1$; in other words $L_6=\G_6$ and $\rho\colon\Delta_6\to \G_6$ is an isomorphism. Finally, we have proved that $L_6$ has index $24$ in $L=\rho(\Delta)$, and by construction $\Delta_6$ has index $24$ in $\Delta$. Therefore $L=\G_W$ (since $[\G_W: \G_6]=[\G_W:L_6]=
 24$) and $\rho\colon\Delta\to\G_W$ is surjective. Thus we have a surjection $\wh{\G}_W\cong\wh{\Delta}\to \wh{\G}_W$, and using the Hopf property again, we conclude that $\wh{\rho}$ is injective, hence so is $\rho$. Thus $\rho\colon\Delta\to\G_W$ is an isomorphism.\qed

\subsection{Profinite rigidity of Kleinian groups containing $\G_W$}\label{friendsofweeks}

As in the case of $\G=\PSL(2,\Z[\omega])$, once the profinite rigidity of $\G_W$ has been established we can deduce the profinite rigidity for other Kleinian groups.  We only record two cases here.  Recall from \S \ref{weeks} that $\Gamma_W$ is a normal subgroup of index $3$ in a group $\G_\mathcal{O}^1$, which we now denote by $\G_1$. Moreover, as noted in \S \ref{weeks}, there is a maximal group $\G_\mathcal{O}$ containing $\G_1$ with $\G_\mathcal{O}/\G_1\cong \Z/2\Z \times \Z/2\Z$.   

\begin{theorem}\label{rigidO1}
$\G_\mathcal{O}$ and $\G_1$ are profinitely rigid.
\end{theorem}

\begin{proof} 
We deal with $\G_1$ first.  As in the proof of Theorems \ref{t:rigid2} and \ref{t:rigid4}, if $\Delta$ is a residually finite group with $\wh{\Delta}\cong \wh{\G}_1$, then $\Delta$ contains an index $3$ normal subgroup $\Omega$ with $\wh{\Omega}\cong \wh{\G}_W$, and from the profinite rigidity of $\G_W$ and Mostow Rigidity we deduce that $\Delta$ is either an arithmetic Kleinian group containing $\G_W$ as a normal subgroup of index $3$, or else $\Delta\cong \G_W\times \Z/3\Z$.  The latter case can be excluded for this implies that $\Delta$ and $\G_1$ surject $\Z/5\Z\times \Z/5\Z$, which they do not, because as was pointed out in \S \ref{weeks}, the orbifold $\H^3/\G_1$ is obtained from $(3,0)$ Dehn surgery on $5_2$ and so has abelianization $\Z/3\Z$.

Thus we may assume that $\Delta$ is an arithmetic Kleinian group containing $\G_W$.  (We remark that this forces $\Delta$ to contain an element of finite order since $M_W$ is the minimal-volume closed (arithmetic) hyperbolic 3--manifold; and the only possible order for a torsion element is 3, since $\G_W<\Delta$ is a normal subgroup of index $3$.) Because $\G_1^{\ab}\cong \Z/3\Z$, we know that $H_1(\Delta,\Z/2\Z)=0$, and since there is a unique maximal order in the invariant quaternion algebra (as remarked in \S \ref{weeks}) we can assume that $\Delta < \G_1$. But then $[\G_1:\G_W]=[\Delta:\G_W]=3$, so $\Delta=\G_1$ as required. 

We now deal with $\G_\mathcal{O}$. Suppose that $\Delta$ is a finitely generated, residually finite group with $\widehat{\Delta} = \widehat{\G}_{\mathcal{O}}$. The commutator subgroup of $\G_{\mathcal{O}}$ is $\G_1$, and ${\G}_{\mathcal{O}}/\G_1\cong \Z/2\Z\times \Z/2\Z$. Hence there is a unique normal subgroup  of index $4$ in $\widehat{\G}_{\mathcal{O}}$, which we denote by $\Omega$. By the first part of the proof, $\Delta_1:=\Omega\cap \Delta$ is isomorphic to $\G_1= \Omega\cap {\G}_{\mathcal{O}}$. $\G_{\mathcal{O}}$ is maximal in the commensurability class of $\G_1$ (see \cite[Ch 11]{MR}), so by Mostow Rigidity its action by conjugation on $\G_1$ gives an isomorphism $\G_{\mathcal{O}}\cong{\rm{Aut}}(\G_1)$. 

The action of $\Delta$ by conjugation on  $\Delta_1$  defines a map  $c\colon\Delta\to{\rm{Aut}}(\G_1)\cong\G_{\mathcal{O}}$ whose image contains $\G_1$, the group of inner automorphisms. Thus we obtain a map $c'$ from $\Delta$ to the finite group ${\rm{Out}}(\G_1)=\G_{\mathcal{O}}/\G_1$ with kernel $\Delta_1$. The restriction to $\G_{\mathcal{O}}$ of $\widehat{c'}\colon \wh{\Delta}\to {\rm{Out}}(\G_1)=\G_{\mathcal{O}}/\G_1$ is the standard map. In particular it is surjective, so $c\colon\Delta\to\G_{\mathcal{O}}$ is surjective. And since $c$ restricts to an isomorphism between subgroups of index $4$, it is an isomorphism. 
\end{proof}

\begin{remark} 
A slight adjustment to the second half of the above proof shows that if $\Delta$ is a finitely generated, profinitely rigid, centerless group with ${\rm{Out}}(\Delta)$ finite, then ${\rm{Aut}}(\Delta)$ is profinitely rigid.
\end{remark}

\subsection{Additional full-size examples}\label{full_sized}

$\G_W$ denotes the fundamental group of the Weeks manifold.

\begin{theorem}\label{inf_many}
For all integers $r\geq 0$, the groups $\G_W \times \Z^r$ are profinitely rigid.
\end{theorem}

\begin{proof} 
Theorem \ref{main_weeks} shows that we can assume that $r\geq 1$.  Suppose that $\Delta$ is a finitely generated residually finite group with $\wh{\Delta}\cong \wh{\G_W \times \Z^r}\cong\wh{\G}_W \times \wh{\Z}^r$.  Since $\Z^r$ is central, any Zariski-dense representation of $\G_W \times \Z^r$ into $\PSL(2,\C)$ must kill $\Z^r$. (It is important here to work with the centerless $\PSL(2,\C)$.) Hence there is a bijection between $\mathrm{X}_{\mathrm{zar}}(\G_W,\mathbb{C})$ and $\mathrm{X}_{\mathrm{zar}}(\G_W\times \Z^r,\mathbb{C})$. By Corollary \ref{weeks_galois_rigid},  $|\mathrm{X}_{\mathrm{zar}}(\G\times \Z^r,\mathbb{C})|=3$. It follows that $\G_W\times \Z^r$ is Galois rigid, and so by Theorem \ref{T1}, $\Delta$ is Galois rigid. From Galois rigidity, we get an epimorphism $\rho\colon\Delta \rightarrow L$ where $L<\Gamma_{\mathcal{O}^1}$ and $\mathcal{O}$ is a maximal order contained in the invariant quaternion algebra of $\G_W$. Recall that  we can conjugate $\G_W$ into $\Gamma_{\mathcal{O}^1}$, in which case
$[\Gamma_{\mathcal{O}^1}:\G_W]=3$.  By construction, the kernel of $\widehat{\rho}\colon\widehat{\Delta}\to\widehat{L}$ contains $\{1\}\times \widehat{\Z^r}$. Thus $\hat{\rho}$ factors through the projection to the first factor of $\widehat{\G_W}\times \widehat{\Z^r}$ and we obtain a continuous epimorphism $\widehat{\rho_0}\colon\widehat{\G_W}\to \widehat{L}$. It follows that $\wh{L}$ cannot have $\Z/3\Z$ as a quotient, and hence there is no $3$--torsion in $H_1(L,\Z)$. The argument of Lemma \ref{l: reduce_weeks} now applies to show that $L<\G_W$.   

We may now run the argument in the proof of Theorem \ref{main_weeks} on $\wh{\G}_W\rightarrow \wh{L} \hookrightarrow \wh{\G}_W$ to deduce that $L=\G_W$ and $\widehat{\rho_0}$ is an isomorphism. Thus we obtain a short exact sequence
\begin{equation}\label{central}
1\to \Omega \to \Delta \to \G_W\to 1,
\end{equation}
where $\Omega=\Delta \cap (\{1\} \times \widehat{\Z^r})$ is central in $\Delta$. Associated to any short exact sequence of groups $1\to G_1\to G_2\to G_3\to 1$ one has a 5--term exact sequence in homology with $\Z$--coefficients,
\[ H_2(G_3,\Z)\to H_0(G_3, H_1(G_1,\Z)) \to H_1(G_2,\Z) \to H_1(G_3,\Z) \to 0; \]
for central extensions the second term is simply $G_1$. In the case $G_3=\G_W$, from \S \ref{weeks} we have $H_1(\G_W,\Z)=(\Z/5\Z)^2$,  and $H_2(\G_W,\Z)=H^1(\G_W,\Z)=0$ by Poincar\'e duality. Thus the 5--term sequence reduces to a short exact sequence 
\begin{equation}\label{last-eq}
0\to \Omega \to H_1(\Delta,\Z)  \to (\Z/5\Z)^2\to 0,
\end{equation}
where $\Omega$ is torsion-free. But $\widehat{\Delta} = \widehat{\G_W\times \Z^r}$, so $H_1(\Delta,\Z) = (\Z/5\Z)^2 \times \Z^r$, by Lemma  \ref{l:abel}. Thus $\Omega=\Z^r$ and (\ref{last-eq}) splits. It follows that the central extension (\ref{central}) also splits, and therefore $\Delta\cong \G_W \times \Z^r$ as claimed. 
\end{proof}

\bigskip


\noindent{\bf Appendix:}\\[\baselineskip]

Here we prove Lemma \ref{L4}, whose statement we repeat for the convenience of the reader.

\medskip

\noindent{\bf Lemma A.1.}~{\em  If $K$, $K'$ are number fields and $\tau_f\colon V_{K'}^f \to V_K^f$ is an injective map with $K'_w \cong K_{\tau_f(w)}$ for all $w \in V_{K'}^f$, then $\tau_f$ is a bijection. }

\begin{proof}
As $K'_w \cong K_{\tau_f(w)}$ for all $w \in V_{K'}^f$, any prime that splits completely in $K$ must also split completely in $K'$. By \cite[p.~108, Cor to Thm 31]{Mar} and \cite[p.~164, Thm 5.5]{Jan}, we see that $K'$ is a subfield of the Galois closure $K_{\mathrm{gal}}$ of $K$ over $\mathbb{Q}$. Let $G = \mathrm{Gal}(K_{\mathrm{gal}}/\mathbb{Q})$, and let $H$ and $H'$ be the subgroups of $G$ that fix $K$ and $K'$ respectively. Let $\chi, \chi'$ be the permutation characters for the $G$ action on $G/H$ and $G/H'$. By \cite[p.~128, Prop 2.7]{Jan}, if $p$ is a rational prime that does not ramify in $K_{\mathrm{gal}}$ and $g \in G$ is the Frobenius automorphism of any prime of $K_{\mathrm{gal}}$ over $p$, then the inertia degrees of the primes over $p$ in $K, K'$ are given by the orders of the orbits of $\langle g \rangle$ acting on $G/H$ and $G/H'$. Restricting $\tau_f$ to $V_{K'}^p$, we get an injection $G/H' \to G/H$ of $\langle g \rangle$--sets. Thus $\chi'(g) \leq \chi(g)
 $. By Chebotarev's density theorem, each $g \in G$ is the Frobenius automorphism for infinitely many unramified primes of $K_{\mathrm{gal}}$. Therefore, 
\begin{equation}\label{Eq:CharEq0}
\chi'(g) \leq \chi(g)
\end{equation}
 for every $g \in G$. For a $G$--set $Z$ with permutation character $\chi_Z$, the number of orbits of the action on $Z$ is $\frac{1}{\abs{G}} \sum_g \chi_Z(g)$. Hence
\[ \frac{1}{\abs{G}} \sum_g \chi(g) = \frac{1}{\abs{G}} \sum_g \chi'(g) = 1 \] 
and $\sum_g (\chi(g) - \chi'(g)) = 0$. From (\ref{Eq:CharEq0}) we deduce $\chi = \chi'$. Thus $\abs{G/H} = \abs{G/H'}$ and $K,K'$ have the same degree over $\Q$, so $\tau_f$ is a bijection.
\end{proof}


\newpage 

\begin{center}
\textbf{Supplemental Magma Code}
\end{center}

\smallskip

\smallskip

\smallskip

\smallskip

\smallskip

\smallskip

\setcounter{section}{0}

\section{Introduction}

This document is a supplement to our paper \cite{BMRS}. It contains the Magma codes used to verify several group theoretic facts needed in the proofs of the main results of \cite{BMRS}. Before giving the Magma codes and their output, we recall some notation from \cite{BMRS}.

In what follows, $\Lambda_0$ is the group of isometries of hyperbolic 3--space generated by reflections in the faces of the tetrahedron $T_0 = T[3,2,2;6,2,3]$. The index $2$ subgroup of $\Lambda_0$ formed by the orientation-preserving isometries is $\Gamma_0 \cong \mathrm{PGL}(2,\Z[\omega])$. 

The group $\Lambda_1$ is generated by reflections in the face of the tetrahedron $T_1=T[3,2,2;3,3,3]$; this is an index $2$ subgroup of $\Lambda_0$. The index $2$ subgroup of $\Lambda_1$ formed by the orientation-preserving isometries is $\Gamma =\mathrm{PSL}(2,\Z[\omega])$. 

The group $\Lambda_0$ has a third subgroup of index $2$ that will be denoted by $\Lambda_2$. One can check that $\Lambda_2$ contains $\G$ as a subgroup of index $2$. 

\begin{equation}\label{Eq:Natural}
\begin{tikzcd}
& & \Lambda_0 \arrow[rrdd,"2",dash] \arrow[lldd,"2"',dash] \arrow[dd,"2"',dash] & & \\ & & & & \\ \Lambda_1 \arrow[rrdd,"2"', dash] & & \Gamma_0 \arrow[dd,"2"', dash] &  &\Lambda_2 \arrow[lldd,"2", dash] \\ & & & & \\ & & \Gamma & &
\end{tikzcd}
\end{equation}

\noindent The groups $\Lambda_0$, $\Gamma_0$, $\Lambda_1$, $\Gamma$, and $\Lambda_2$ can be presented as follows:
\begin{align*}
\Lambda_0 &= \innp{ x,y,z,w~\mid~ x^2 = y^2 = z^3 = w^2 = (xy)^3 = (xz)^2 = (xw)^2 = (yz)^3 = (yw)^2 = (zw)^6 = 1} \\
\G_0 &= \innp{ x,y,z ~\mid~ x^3= y^2= z^2 = (yz)^6 = (zx^{-1})^2 = (yx^{-1})^3 =1} \\
\Lambda_1 &= \innp{ x,y,z,w ~\mid~ x^2 = y^2 = z^2 = w^2 = (xy)^2 = (xz)^2 = (xw)^3 = (yz)^3 = (yw)^3 = (zw)^3 =1} \\
\G &= \innp{ x,y,z ~\mid~ x^3 = y^2 = z^2 = (yz)^3 = (zx^{-1})^3 = (yx^{-1})^3 = 1} \\
\Lambda_2 &= \innp{x,y,z ~\mid~ y^2 = x^2 = z^3 = zyz^{-1}y = (zx)^3 = (xy)^6 = 1}.
\end{align*}

\noindent Finally, $\Gamma_W$ is the fundamental group of the Weeks manifold; it has the following presentation

\[ \G_W=\innp{ a,b~\mid~ababa^{-1}b^2a^{-1}b = abab^{-1}a^2b^{-1}ab = 1}. \]

\noindent We remind the reader that Magma uses the notation $[m,n,0]$ to declare that the abelianization of a group is $\Z/m\Z\times\Z/n\Z\times \Z$.

\section{The proof of \cite[Lemma 7.2]{BMRS}}\label{magma_bianchi}

What follows is the Magma routine to complete the proof of \cite[Lemma 7.2]{BMRS}. Here, $\tt{g}$ is the group $\Gamma$, presented as in (\ref{Eq:Natural}). The routine computes subgroups of index $n$ less than or equal to $12$, prints the number of such and computes their abelianizations. 

\medskip

\begin{verbatim}
> g<a,b,c>:=Group<a,b,c|a^3,b^2,c^2,(b*c)^3,(c*a^-1)^3,(b*a^-1)^3>;
> print AbelianQuotientInvariants(g);
[ 3 ]
> l:=LowIndexSubgroups(g,<3,3>); 
> print #l;
1
> print AbelianQuotientInvariants(l[1]);
[ 2, 2 ]
> l:=LowIndexSubgroups(g,<4,4>); 
> print #l;                             
1
> print AbelianQuotientInvariants(l[1]);
[ 3, 3 ]
> l:=LowIndexSubgroups(g,<5,5>); 
> print #l;                             
1
> print AbelianQuotientInvariants(l[1]);
[ 3, 3 ]
> l:=LowIndexSubgroups(g,<6,6>); 
> print #l;                             
2
> print AbelianQuotientInvariants(l[1]);
[ 2, 0 ]
> print AbelianQuotientInvariants(l[2]);
[ 6 ]
> l:=LowIndexSubgroups(g,<7,7>); 
> print #l;                             
4
> print AbelianQuotientInvariants(l[1]);
[ 6 ]
> print AbelianQuotientInvariants(l[2]);
[ 6 ]
> print AbelianQuotientInvariants(l[3]);
[ 6 ]
> print AbelianQuotientInvariants(l[4]);
[ 6 ]
> l:=LowIndexSubgroups(g,<8,8>);        
> print #l;                             
2
> print AbelianQuotientInvariants(l[1]);
[ 3, 3 ]
> print AbelianQuotientInvariants(l[2]);
[ 3, 3 ]
> l:=LowIndexSubgroups(g,<9,9>); 
> print #l;                             
1
> print AbelianQuotientInvariants(l[1]);
[ 2, 2 ]
> l:=LowIndexSubgroups(g,<10,10>); 
> print #l;                             
1
> print AbelianQuotientInvariants(l[1]);
[ 6, 0 ]
> l:=LowIndexSubgroups(g,<11,11>); 
> print #l;                             
0
> l:=LowIndexSubgroups(g,<12,12>); 
> print #l;                       
7
> print AbelianQuotientInvariants(l[1]);
[ 0 ]
> print AbelianQuotientInvariants(l[2]);
[ 5, 0 ]
> print AbelianQuotientInvariants(l[3]);
[ 3, 9 ]
> print AbelianQuotientInvariants(l[4]);
[ 3, 9 ]
> print AbelianQuotientInvariants(l[5]);
[ 3, 3, 3 ]
> print AbelianQuotientInvariants(l[6]);
[ 2, 0 ]
> print AbelianQuotientInvariants(l[7]);
[ 3, 3, 3 ]
\end{verbatim}


\section{The proof of \cite[Proposition 7.3]{BMRS}}\label{magma_sister5}

What follows is the Magma routine to complete the proof of \cite[Proposition 7.3]{BMRS}. This runs through all of the index $5$ subgroups in the groups under consideration and calculates their abelianizations. We again use the presentation of $\G$ from (\ref{Eq:Natural}). We
consider those subgroups listed in the previous Magma routine that have first Betti number $b_1=1$, except that we ignore the index $12$ subgroup of $\G$ with abelianization $\Z$, because this is the fundamental group of the figure-eight knot complement, which can be eliminated from consideration because its unique $5$--fold cyclic cover is a once-punctured torus bundle over the circle with $b_1=1$. The last group whose index $5$ subgroups are analysed below is the fundamental group of the sister of the figure-eight knot complement.

\medskip

\begin{verbatim}
g<a,b,c>:=Group<a,b,c|a^3,b^2,c^2,(b*c)^3,(c*a^-1)^3,(b*a^-1)^3>;
> l:=LowIndexSubgroups(g,<6,6>); 
> print #l;
2
> print AbelianQuotientInvariants(l[1]);
[ 2, 0 ]
> h:=Rewrite(g,l[1]); 
> k:=LowIndexSubgroups(h,<5,5>); 
> print #k;
4
> print AbelianQuotientInvariants(k[1]);
[ 2, 0, 0 ]
> print AbelianQuotientInvariants(k[2]);
[ 2, 0 ]
> print AbelianQuotientInvariants(k[3]);
[ 2, 0, 0, 0 ]
> print AbelianQuotientInvariants(k[4]);
[ 2, 0, 0 ]
> l:=LowIndexSubgroups(g,<10,10>);
> print #l;
1
> print AbelianQuotientInvariants(l[1]);
[ 6, 0 ]
> h:=Rewrite(g,l[1]); 
> k:=LowIndexSubgroups(h,<5,5>); 
> print #k;                             
6
> print AbelianQuotientInvariants(k[1]);
[ 3, 6, 0 ]
> print AbelianQuotientInvariants(k[2]);
[ 3, 6, 0 ]
> print AbelianQuotientInvariants(k[3]);  
[ 3, 6, 0 ]
> print AbelianQuotientInvariants(k[4]);  
[ 6, 0 ]
> print AbelianQuotientInvariants(k[5]);
[ 2, 6, 0, 0 ]
> print AbelianQuotientInvariants(k[6]);
[ 3, 6, 0, 0 ]
> l:=LowIndexSubgroups(g,<12,12>); 
> print #l;
7
> print AbelianQuotientInvariants(l[6]);
[ 2, 0 ]
> h:=Rewrite(g,l[6]); 
> k:=LowIndexSubgroups(h,<5,5>); 
> print #k;                             
4
> print AbelianQuotientInvariants(k[1]);
[ 2, 0, 0, 0 ]
> print AbelianQuotientInvariants(k[2]);
[ 2, 0, 0, 0 ]
> print AbelianQuotientInvariants(k[3]);
[ 2, 0, 0, 0 ]
> print AbelianQuotientInvariants(k[4]);  
[ 2, 0 ]
> print AbelianQuotientInvariants(l[2]);
[ 5, 0 ]
> h:=Rewrite(g,l[2]); 
> k:=LowIndexSubgroups(h,<5,5>); 
> print #k;                             
8
> print AbelianQuotientInvariants(k[1]);
[ 0, 0, 0, 0, 0 ]
> print AbelianQuotientInvariants(k[2]);
[ 5, 5, 0 ]
> print AbelianQuotientInvariants(k[3]);
[ 5, 5, 0 ]
> print AbelianQuotientInvariants(k[4]);  
[ 5, 5, 0 ]
> print AbelianQuotientInvariants(k[5]);
[ 5, 25, 0 ]
> print AbelianQuotientInvariants(k[6]);
[ 0, 0, 0 ]
> print AbelianQuotientInvariants(k[7]);
[ 0, 0, 0 ]
> print AbelianQuotientInvariants(k[8]);
[ 5, 5, 0 ]
\end{verbatim}


\section{The proof of \cite[Theorem 8.2]{BMRS}}\label{magma_lam2}

The following routine is used in the proof of \cite[Theorem 8.2]{BMRS}, where we promote the profinite rigidity of $\G$ to profinite rigidity for each of its index $2$ extensions,  $\G_0$, $\Lambda_1$, $\Lambda_2$ and $\G \times \Z/2\Z$. The group $\tt{g}$ below is $\Lambda_0$, presented as in (\ref{Eq:Natural}), while $\tt{z}$ is the group $\Lambda_2$, presented by applying ``Rewrite" to the group $\tt{l}[2]$, and $\tt{bia}$ is the group $\G \times \Z/2\Z$. 

\medskip

\begin{verbatim}
> g<r1,r2,r3,r4>:=Group<r1,r2,r3,r4|r1^2,r2^2,r3^2,r4^2,(r1*r2)^3,(r1*r3)^2,(r1*r4)^2,
> (r2*r3)^3,(r2*r4)^2,(r3*r4)^6>; 
> print AbelianQuotientInvariants(g);   
[ 2, 2 ]
> l:=LowIndexSubgroups(g,<2,2>); 
> print AbelianQuotientInvariants(l[1]);
[ 2 ]
> print AbelianQuotientInvariants(l[2]); 
[ 6 ]
> print AbelianQuotientInvariants(l[3]);
[ 2 ]
> print l[1];
Finitely presented group on 3 generators, Index in group g is 2,
Generators as words in group g
    $.1 = r2 * r1, $.2 = r3 * r1, $.3 = r4 * r1
> print l[2];
Finitely presented group on 3 generators, Index in group g is 2,
Generators as words in group g
    $.1 = r2 * r1, $.2 = r3 * r1, $.3 = r4
> print l[3];
Finitely presented group on 4 generators, Index in group g is 2, 
Generators as words in group g
    $.1 = r1, $.2 = r2, $.3 = r3, $.4 = r4 * r3 * r4
> z:=Rewrite(g,l[2]); print z;
Finitely presented group z on 3 generators, Generators as words in group g
    z.1 = r2 * r1, z.2 = r3 * r1, z.3 = r4
Relations
    z.3^2 = Id(z), z.2^2 = Id(z), z.1^3 = Id(z), z.1 * z.3 * z.1^-1 * z.3 = Id(z)
    (z.1 * z.2)^3 = Id(z), (z.2 * z.3)^6 = Id(z)    
> bia<a,b,c,t>:=Group<a,b,c,t|(t,a),(t,b),(t,c),t^2,a^3,b^2,c^2,(b*c)^3,
(c*a^-1)^3,(b*a^-1)^3>; 
> q:=LowIndexSubgroups(z,<7,7>); 
> print #q;
0
> qq:=LowIndexSubgroups(bia,<7,7>); 
> print #qq;                       
4    
\end{verbatim}   

\medskip

\noindent The proof of \cite[Theorem 8.2]{BMRS} also relies on the following routine which distinguishes $\widehat{\G}_0$ from $\widehat{\Lambda}_1$ by counting conjugacy classes of index $8$ subgroups.

\medskip

\begin{verbatim}
> g<x,y,z,w>:=Group<x,y,z,w|x^2,y^2,z^2,w^2,(x*y)^2,(x*z)^2,(x*w)^3,
(y*z)^3,(y*w)^3,(z*w)^3>; l:=LowIndexSubgroups(g,<8,8>); 
print #l; 
1
> h<a,b,c>:=Group<a,b,c|a^3,b^2,c^2,(b*c)^6,(c*a^-1)^2,(b*a^-1)^3>;
> l:=LowIndexSubgroups(h,<8,8>); 
print #l;
3
\end{verbatim}

\section{The proof of \cite[Theorem 8.3]{BMRS}} 

\cite[Theorem 8.3]{BMRS} relies on the following Magma routine which shows that $\tt{g}=\Lambda_0$ has fewer index $8$ subgroups, up to conjugacy, than ${\tt{bia}} =\G_0\times \Z/2\Z$.

\medskip

\begin{verbatim}
> g<r1,r2,r3,r4>:=Group<r1,r2,r3,r4|r1^2,r2^2,r3^2,r4^2,(r1*r2)^3,(r1*r3)^2,
(r1*r4)^2, (r2*r3)^3,(r2*r4)^2,(r3*r4)^6>; bia<a,b,c,t>:=Group<a,b,c,t|(t,a),
(t,b),(t,c),t^2,a^3,b^2,c^2,(b*c)^6,(c*a^-1)^2, (b*a^-1)^3>;
> l:=LowIndexSubgroups(g,<8,8>);  
k:=LowIndexSubgroups(bia,<8,8>); 
print #l;                       
3
> print #k;                       
5
\end{verbatim}

\section{Calculations for $\G_W$} \label{magma_weeks} 

This section contains the Magma calculations used in \cite[\S 9]{BMRS} in proving that the fundamental group of the Weeks manifold is profinitely rigid; see in particular the proof of \cite[Lemma 9.3]{BMRS}. The group ${\tt{g}}$ in this routine is $\Gamma_W$, given by the presentation from (\ref{Eq:Natural}).

\medskip

\begin{verbatim}
> g<a,b>:=Group<a,b| a*b*a*b*a^-1*b^2*a^-1*b,a*b*a*b^-1*a^2*b^-1*a*b>;      
> print AbelianQuotientInvariants(g);
[ 5, 5 ]
> l:=LowIndexSubgroups(g,<24,24>); 
> print #l;             
11
> print AbelianQuotientInvariants(l[1]);
[ 5, 55, 0 ]
> f,i:=CosetAction(g,l[1]); 
> print Order(i);
6072
> IsSimple(i);
true
> print AbelianQuotientInvariants(l[2]);
[ 2, 2, 2, 10, 110 ]
> f,i:=CosetAction(g,l[2]); 
> print Order(i);                       
6072
> IsSimple(i);                          
true
> print AbelianQuotientInvariants(l[3]);
[ 5, 30, 0 ]
> f,i:=CosetAction(g,l[3]); 
> print Order(i);                       
2204496
> print AbelianQuotientInvariants(l[4]);
[ 90, 90 ]
> f,i:=CosetAction(g,l[4]); 
> print Order(i);                       
2204496
> print AbelianQuotientInvariants(l[5]);
[ 5, 30, 0 ]
> f,i:=CosetAction(g,l[5]); 
> print Order(i);                       
2204496
> print AbelianQuotientInvariants(l[6]);
[ 2, 2, 2, 70, 70 ]
> f,i:=CosetAction(g,l[6]); 
> print Order(i);                       
168
> print AbelianQuotientInvariants(l[7]);  
[ 90, 90 ]
> f,i:=CosetAction(g,l[7]); 
> print Order(i);                       
2204496
> print AbelianQuotientInvariants(l[8]);
[ 5, 5, 10 ]
> f,i:=CosetAction(g,l[8]); 
> print Order(i);
1320
> print AbelianQuotientInvariants(l[9]);
[ 5, 30 ]
> f,i:=CosetAction(g,l[9]); 
> print Order(i);                       
310224200866619719680000
> print AbelianQuotientInvariants(l[10]); 
> f,i:=CosetAction(g,l[10]); 
> print Order(i);           
310224200866619719680000
> print AbelianQuotientInvariants(l[11]);
[ 5, 30 ]
> f,i:=CosetAction(g,l[11]); 
> print Order(i);                        
310224200866619719680000
\end{verbatim}

\medskip

\noindent The following enumeration of index $8$ subgroups of $\G_W$ is also required in the proof of \cite[Lemma 9.3]{BMRS}.  Again $\tt{g}$ denotes $\G_W$.

\medskip

\begin{verbatim}
> g<a,b>:=Group<a,b| a*b*a*b*a^-1*b^2*a^-1*b,a*b*a*b^-1*a^2*b^-1*a*b>;                                                                                                                                  
> print AbelianQuotientInvariants(g);
[ 5, 5 ]
> l:=LowIndexSubgroups(g,<8,8>); print #l;
1
> print AbelianQuotientInvariants(l[1]);
[ 5, 30 ]
\end{verbatim}



\begin{thebibliography}{9999}

\bibitem{Ad} 
C.~C.~Adams, \emph{Non-compact hyperbolic 3-orbifolds of small volume}, Topology '90, 1--15, Ohio State Univ. Math. Res. Inst. Publ. {\bf 1}, de Gruyter, Berlin, (1992).

\bibitem{Ag}
I.~Agol, \emph{The virtual Haken conjecture}, with an appendix by I.~Agol, D.~Groves, J.~Manning, Documenta Math. {\bf 18} 
(2013), 1045--1087.

\bibitem{Ag2} 
I.~Agol, \emph{Virtual properties of 3-manifolds}, Proceedings of the International Congress of Mathematicians-Seoul 2014 {\bf 1}, 141--170, Kyung Moon Sa, Seoul, (2014).

\bibitem{Ak} 
M.~Aka, \emph{Arithmetic groups with isomorphic finite quotients}, J. Algebra {\bf 352}, (2012), 322--340.

\bibitem{Alp} 
R.~C.~Alperin, \emph{Normal subgroups of $\PSL(2,\Z[\sqrt{-3}])$}, Proc. Amer. Math. Soc {\bf 124} (1996), 2935--2941.

\bibitem{BaR} 
M.~D.~Baker, A.~W.~Reid, \emph{Arithmetic knots in closed 3-manifolds}, J. Knot Theory Ramifications, {\bf 11} (2002), 903--920.

\bibitem{Baum74} 
G.~Baumslag, \emph{Residually finite groups with the same finite images}, Compositio Math. {\bf 29} (1974), 249--252.

\bibitem{BF} 
M.~Boileau, S.~Friedl, \emph{The profinite completion of 3-manifold groups, fiberedness and the Thurston norm}, to appear in What's Next?: The Mathematical Legacy of William P. Thurston, P.U.P. (arXiv:1505.07799)

\bibitem{Bos} 
W.~Bosma, J.~Cannon, C.~Playoust, \emph{The Magma algebra system. I. The user language}, J. Symbolic Comput.,  {\bf 24} (1997), 235--265.

\bibitem{BZ} 
S.~Boyer, X.~Zhang, \emph{On Culler-Shalen seminorms and Dehn filling}, Ann. of Math. {\bf 148} (1998), 737--801.

\bibitem{BCR} 
M.~R.~Bridson, M.~Conder, A.~W.~Reid, \emph{Determining Fuchsian groups by their finite quotients},  Israel J. Math. {\bf 214} (2016), 1--41.

\bibitem{BMRS2}
M.~R.~Bridson, D.~B.~McReynolds, A.~W.~Reid, R.~Spitler, \emph{On the profinite rigidity of triangle groups}, arXiv:2004.07137.

\bibitem{BridR} 
M.~R.~Bridson, A.~W.~Reid, \emph{Profinite rigidity, fibering and the figure-eight knot}, to appear in What's Next?: The Mathematical Legacy of William P. Thurston, P.U.P. (arXiv:1505.07886)

\bibitem{BRW} 
M.~R.~Bridson, A.~W.~Reid, H.~Wilton, \emph{Profinite rigidity and surface bundles over the circle}, Bull. L. M. S. {\bf 49} (2017) 831--841.

\bibitem{BLW} 
A.~M.~Brunner, Y.~W.~Lee, N.~J.~Wielenberg, \emph{Polyhedral groups and graph amalgamation products}, Topology Appl. {\bf 20} (1985), 289--304. 

\bibitem{Butt} 
J.~O.~Button, \emph{Fibred and virtually fibred hyperbolic 3-manifolds in the censuses}, Experiment. Math. {\bf 14} (2005), 23--255.

\bibitem{CalDW} 
P.~J.~Callahan, J.~C.~Dean, J.~R.~Weeks, \emph{The simplest hyperbolic knots}, J. Knot Theory Ramifications {\bf 8} (1999),  279--297. 

\bibitem{Knot} 
J.~C.~Cha, C.~Livingston, \emph{KnotInfo: Table of Knot Invariants}, http://www.indiana.edu/$\sim$knotinfo, September 2018.

\bibitem{CFJR} 
T.~Chinburg, E.~Friedman, K.~N.~Jones, A.~W.~Reid, \emph{The smallest volume arithmetic hyperbolic 3-manifold}, Ann. Sc. Norm. Super. Pisa Cl. Sci. {\bf 30} (2001), 1--40.

\bibitem{CHLR} 
T.~Chinburg, E.~Hamilton, D.~D. Long, A.~W. Reid, \emph{Geodesics and commensurability classes of arithmetic hyperbolic 3--manifolds},  Duke Math. J. \textbf{145} (2008), 25--44.

\bibitem{snap} 
D.~Coulsen, O.~A.~Goodman, C.~D.~Hodgson, W.~D.~Neumann, \emph{Computing arithmetic invariants of 3-manifolds}, Experiment. Math. {\bf 9} (2000), 127--152. 

\bibitem{CDW} 
M.~Culler, N.~M.~Dunfield, M.~Goerner, J.~R.~Weeks, \emph{SnapPy, a computer program for studying the geometry and topology of 3-manifolds}, http://snappy.computop.org.

\bibitem{DFPR} 
J.D.~Dixon, E.W.~Formanek, J.C.~Poland, L.~Ribes, \emph{Profinite completions and isomorphic finite quotients}, J.~Pure Appl.~Algebra {\bf 23} (1982), 227--231.

\bibitem{DuMe} 
W.~D.~Dunbar, G.~R.~Meyerhoff, \emph{Volumes of  hyperbolic 3-orbifolds}, Indiana J. Math {\bf 43} (1994), 611--637.

\bibitem{Fi} 
B.~Fine, \emph{The Algebraic Theory of the Bianchi groups}, Monographs and Textbooks in Pure and Applied Mathematics, {\bf 129} Marcel Dekker, Inc., New York, (1989).

\bibitem{Fu} 
L.~Funar,  \emph{Torus bundles not distinguished by {TQFT} invariants}, Geom. Topol. {\bf 17} (2013), 2289--2344. With an appendix by L.~Funar,  A.~Rapinchuk.

\bibitem{GS} 
F.~J.~Grunewald, R.~Scharlau, \emph{A note on finitely generated torsion-free nilpotent groups of class 2}, J. Algebra {\bf 58} (1979), 162--175.

\bibitem{He} 
J.~Hempel, \emph{Some 3-manifold groups with the same finite quotients},  arXiv:1409.3509.

\bibitem{Iwasawa} 
K.~Iwasawa, \emph{On the ring of valuation vectors}, Annals of Math. \textbf{57} (1953), 331--356.

\bibitem{Jan} 
G.~J.~Janusz, \emph{Algebraic Number Fields, Second Edition}, Graduate Studies in Math. {\bf 7} A.M.S. (1996).

\bibitem{Kan} 
T.~Kanenobu, \emph{The augmentation subgroup of a pretzel link}, Math. Sem. Notes Kobe Univ. {\bf 7} (1979), 363--384.

\bibitem{MR} 
C.~Maclachlan, A.~W.~Reid, \emph{The Arithmetic of Hyperbolic 3-Manifolds}, Graduate Texts in Math. {\bf 219}, Springer-Verlag (2003).

\bibitem{Mar} 
D.~A.~Marcus, \emph{Number Fields}, Springer-Verlag (1996).

\bibitem{Rem}  
V.~D.~Mazurov, E.~I.~Khukhro (editors), \emph{Unsolved problems in group theory}, The Kourovka Notebook {\bf 17}, Novosibirsk (2010).

\bibitem{McR} 
D.~B.~McReynolds, \emph{Peripheral separability and cusps of arithmetic hyperbolic orbifolds}, Algebr. Geom. Topol. {\bf 4} (2004), 721--755. 

\bibitem{McSpit} 
D.~B.~McReynolds, R.~Spitler, \emph{Profinite completions of linear groups and rigid representation}, in preparation. 

\bibitem{MedV}   
A.~Mednykh, A.~Vesnin, \emph{Visualization of the isometry group action on the Fomenko-Matveev-Weeks manifold}, J. Lie Theory, {\bf 8} (1998), 51--66.

\bibitem{Mey} 
G.~R.~Meyerhoff, \emph{The cusped hyperbolic 3-orbifold of minimum volume},  Bull. Amer. Math. Soc. {\bf 13} (1985), 154--156.

\bibitem{NS} 
N.~Nikolov, D.~Segal, \emph{On finitely generated profinite groups.~I. Strong completeness and uniform bounds},  Annals of Math.~{\bf 165} (2007), 171--238. 

\bibitem{PZ} 
L.~Paoluzzi, B.~Zimmermann, \emph{Finite quotients of the Picard group and related hyperbolic tetrahedral and Bianchi groups}, Dedicated to the memory of Marco Reni. Rend. Istit. Mat. Univ. Trieste {\bf 32} (2001), 257--288.

\bibitem{Per} 
R.~Perlis, \emph{On the equation $\zeta_K(s)=\zeta_{K'}(s)$}, J. Number Theory {\bf 9} (1977), 342--360.

\bibitem{Rag}
M.~S.~Raghunathan, \emph{Discrete Subgroups of Lie Groups,} Springer-Verlag (1972).

\bibitem{ReWan} 
A.~W.~Reid, S.~Wang, \emph{Non-Haken 3-manifolds are not large with respect to mappings of non-zero degree}, Comm. Anal. Geom. {\bf 7} (1999), 105--132.

\bibitem{RZ}
L.~Ribes, P.~A.~Zalesskii, \emph{Profinite Groups}, Ergeb.~der Math.~{\bf 40}, Springer-Verlag (2000).

\bibitem{Scott} 
G.~P.~ Scott, \emph{Compact submanifolds of 3-manifolds}, J. London Math. Soc. {\bf 7} (1973), 246--250.

\bibitem{segal}  
D.~Segal, \emph{Polycyclic Groups}, Cambridge University Press (1983).

\bibitem{Serre} 
J-P.~Serre, \emph{Trees}, Springer-Verlag (1980).

\bibitem{Spitler}
R.~Spitler, \emph{Profinite Completions and Representations of Finitely Generated Groups}, PhD thesis, Purdue University (2019).

\bibitem{Th} 
W.~P.~Thurston, \emph{The Geometry and Topology of 3-Manifolds}, Princeton University mimeographed notes, (1979).

\bibitem{Vin} 
E.~B.~Vinberg, \emph{The smallest field of definition of a subgroup of $\PSL_2$}, Russian Acad. Sc. Sb. Math. {\bf 80} (1995), 179--190.

\bibitem{Wei} 
B.~Weisfeiler, \emph{Strong approximation for Zariski dense subgroups of semi-simple algebraic groups}, Annals of Math. {\bf 120} (1984), 271--315.

\bibitem{Wilk} 
G.~Wilkes, \emph{Profinite rigidity for {S}eifert fibre spaces}, Geom. Dedicata {\bf 188} (2017), 141--163.

\bibitem{WZ} 
H.~Wilton, P.~A.~Zalesskii, \emph{Distinguishing geometries using finite quotients}, Geom. Topol. {\bf 21} (2017), 345--384.

\bibitem{magma_calcs} Supplemental Magma code to accompany this paper.

\end{thebibliography}

\medskip



\begin{thebibliography}{9999}

\bibitem{BMRS}
M.~R.~Bridson, D.~B.~McReynolds, A.~W.~Reid, R.~Spitler, \emph{Absolute profinite rigidity and hyperbolic geometry}, to appear in Annals of Math.

\end{thebibliography}
\end{document}